\documentclass[12pt, reqno]{amsart}
\usepackage[margin=1in]{geometry}
\usepackage{amsmath, amsfonts, amssymb, url}
\usepackage{eucal}
\usepackage{mathrsfs}
\usepackage{latexsym}
\usepackage{cite}
\usepackage{cases}
\usepackage{xcolor}
\usepackage{enumitem}
\allowdisplaybreaks[4]
\newtheorem{thm}{Theorem}[section]
\newtheorem{prop}{Proposition}[section]

\newtheorem{lem}[thm]{Lemma}
\newtheorem{cor}[thm]{Corollary}

\theoremstyle{remark}

\theoremstyle{definition}

\numberwithin{equation}{section}

\begin{document}
\title[ $(p,q)$-weighted dual curvature measures]
{ $(p,q)$-weighted dual curvature measures}

\author[Q. Huang]{Qingzhong Huang}
	\address[Q. Huang]{College of Data Science, Jiaxing University, Jiaxing 314001,	China} \email{hqz376560571@163.com} 

	\author[A.-J. Li]{Ai-Jun Li}
	\address[A.-J. Li]{School of Science,
		Zhejiang University of Science and Technology, Hangzhou 310023, China}
	\email{liaijun72@163.com}

\begin{abstract}
We study \((p,q)\)-weighted dual curvature measures associated with a fixed nonnegative locally integrable Borel density satisfying a trace-regularity condition. These measures arise naturally from a variational formula for Wulff shapes. We establish the trace-regularity condition when \(g\) is pointwise \(\alpha\)-homogeneous for some \(\alpha\in\mathbb R\), or when \(g\) is radially non-increasing and continuous. For nontrivial even densities in either class, we prove existence results for the even version of the associated Minkowski problem under suitable assumptions. We also establish uniqueness when the prescribed measure has finite support, under a strict positivity assumption on the density and additional restrictions on the parameters. In the case of Lebesgue measure, our curvature measures reduce to those introduced by Lutwak, Yang, and Zhang, and our variational formulas recover their corresponding variational formula. For Gaussian measure, our results are related to those of Feng, Hu, and Xu, and Feng, Li, and Xu. Finally, we investigate the affine properties of the new measures and the associated mixed volumes.
\end{abstract}



\thanks{The first author was supported by  NSFC (Grant No. 11701219) and the Qin Shen Scholar Program of Jiaxing University. The second author was supported by NSFC (Grant No. 12571146 and 12231006). }

\maketitle
\section{Introduction}

We work in the $n$-dimensional Euclidean space $\mathbb R^n$,
where $n\ge 2$. For $x,y\in\mathbb R^n$, let $x\cdot y$
denote their standard inner product and
$|x|=\sqrt{x\cdot x}$ the Euclidean norm. For
$x\in\mathbb R^n\setminus\{o\}$, we write
$\overline{x}=x/|x|$. Throughout the paper, $B$ denotes the
unit ball centered at the origin, and $S^{n-1}=\partial B$
denotes the unit sphere. We use $V$ to denote
$n$-dimensional Lebesgue volume, $\mathcal H^k$ for
$k$-dimensional Hausdorff measure, and $\partial K$ for
the boundary of a set $K$.

A convex body is a compact convex set with nonempty interior.
Its support function is defined by
\begin{equation*}
h_K(x)=\max\{x\cdot y:y\in K\},
\qquad x\in\mathbb R^n.
\end{equation*}
Let $\mathcal K_o^n$ denote the class of convex bodies
containing the origin in their interiors. The polar body of
$K\in\mathcal K_o^n$ is defined by
\begin{equation*}
K^*=\{x\in\mathbb R^n:x\cdot y\le 1
\text{ for all }y\in K\}.
\end{equation*}

A set $L\subset\mathbb R^n$ containing the origin is
star-shaped with respect to the origin if $[o,x]\subset L$
for every $x\in L$. For a compact star-shaped set $L$, its
radial function is defined by
\[
\rho_L(x)=\max\{\lambda\ge0:\lambda x\in L\},
\qquad x\ne o.
\]
A star body is a compact star-shaped set whose radial function
is continuous and strictly positive on $S^{n-1}$. We denote
by $\mathcal S_o^n$ the class of star bodies, by
$\mathcal S_e^n$ its origin-symmetric members, and by
$\mathcal K_e^n$ the class of origin-symmetric convex bodies
in $\mathcal K_o^n$.

The classical Brunn--Minkowski theory is a cornerstone of
convex geometric analysis. Its central objects are convex
bodies and their associated geometric measures, including
volume, surface area, and curvature measures. Two fundamental
problems in this theory are the isoperimetric problem, which
seeks to minimize surface area among convex bodies of
prescribed volume, and the Minkowski problem, which
characterizes the measures on $S^{n-1}$ that arise as surface
area measures of convex bodies and investigates the
uniqueness of the corresponding bodies. These problems have
fostered deep connections with analysis, partial differential
equations, and optimal transport.

Over the past several decades, the classical theory has been
extended in several important directions. One of these is the
$L_p$ Brunn--Minkowski theory, initiated by Firey in the
1960s, in which classical Minkowski addition is replaced by
an $L_p$-type operation. For $p\ge1$, $t>0$, and
$K,L\in\mathcal K_o^n$, the $L_p$ Minkowski--Firey
combination $K+_p t\cdot L$ is defined by
\begin{equation}\label{lp}
h_{K+_p t\cdot L}^p(x)
=h_K^p(x)+t\,h_L^p(x),
\qquad x\in\mathbb R^n.
\end{equation}
A fundamental contribution of Lutwak \cite{L93} is the
following variational formula:
\begin{equation}\label{e-3}
\lim_{t\to0^+}
\frac{V(K+_p t\cdot L)-V(K)}{t}
=
\frac1p\int_{S^{n-1}}h_L^p(u)\,dS_p(K,u),
\end{equation}
where $S_p(K,\cdot)$ denotes the $L_p$ surface area measure
of $K$, defined by
\begin{equation*}
S_p(K,\eta)
=
\int_{n_K^{-1}(\eta)}
(x\cdot n_K(x))^{1-p}
\,d\mathcal H^{n-1}(x)
\end{equation*}
for every Borel set $\eta\subset S^{n-1}$. Here,
$n_K:\partial'K\to S^{n-1}$ is the Gauss map, defined on
the subset $\partial'K$ of boundary points having a unique
outer unit normal, and
\[
n_K^{-1}(\eta)
=
\{x\in\partial'K:n_K(x)\in\eta\}
\]
is the reverse spherical image of $\eta$. The introduction
of $L_p$ surface area measures has played a central role in
the development of the $L_p$ Brunn--Minkowski theory. For a
comprehensive treatment, see \cite[Chapter 9]{Schneider2014}
and the references therein.

The dual Brunn--Minkowski theory, pioneered by Lutwak
\cite{L75} in the 1970s, provides another important extension
of the classical theory. It revolves around dual mixed
volumes of star bodies. For $q\in\mathbb R$ and
$K,Q\in\mathcal S_o^n$, the normalized $\bar q$-th dual
mixed volume is defined by
\begin{equation}\label{VKQ}
\widetilde V_{\bar q}(K,Q)
=
\frac1n\int_{S^{n-1}}
\left(\frac{\rho_K}{\rho_Q}\right)^{\bar q}(u)
\rho_Q^n(u)\,du,
\end{equation}
where $du$ denotes spherical Lebesgue measure and
\[
s^{\bar q}
=
\begin{cases}
\dfrac1q s^q, & q\ne0,\\[6pt]
\log s, & q=0,
\end{cases}
\qquad s>0.
\]
For $q\ne0$, this normalized quantity differs from the
usual $q$-th dual mixed volume by the factor $1/q$.
In particular, when $Q=B$, the normalized $\bar q$-th dual
volume (or dual quermassintegral) of
$K\in\mathcal S_o^n$ is
\begin{equation*}
\widetilde V_{\bar q}(K)
=
\widetilde V_{\bar q}(K,B)
=
\frac1n\int_{S^{n-1}}\rho_K^{\bar q}(u)\,du.
\end{equation*}
These dual volumes and dual quermassintegrals play a
fundamental role in the dual Brunn--Minkowski theory;
see, for example, \cite[Chapter 9]{Schneider2014}.

In \cite{LYZ2018}, Lutwak, Yang, and Zhang introduced a
family of geometric measures on $S^{n-1}$, called the
$(p,q)$-th dual curvature measures. These measures arise
from the first variation of the normalized $\bar q$-th
dual mixed volume under the $L_p$ Minkowski--Firey
combination. More precisely, for $p,q\in\mathbb R$,
$K,L\in\mathcal K_o^n$, and $Q\in\mathcal S_o^n$, one has
\begin{equation}\label{e1}
\lim_{t\to0}
\frac{
\widetilde V_{\bar q}(K+_p t\cdot L,Q)
-\widetilde V_{\bar q}(K,Q)
}{t}
=
\int_{S^{n-1}}
h_L^{\bar p}(v)\,
d\widetilde C_{p,q}(K,Q,v),
\end{equation}
where $s^{\bar p}$ is defined by the same convention as
$s^{\bar q}$, with $q$ replaced by $p$, and
$K+_p t\cdot L$ is defined in \eqref{eq:2.11} and
\eqref{eq:2.12} for all $p\in\mathbb R$.
The $(p,q)$-th dual curvature measure
$\widetilde C_{p,q}(K,Q,\cdot)$ is defined by
\begin{equation}\label{CqQ}
\widetilde C_{p,q}(K,Q,\eta)
=
\frac1n
\int_{n_K^{-1}(\eta)}
(x\cdot n_K(x))^{1-p}
\rho_Q^{\,n-q}(x)
\,d\mathcal H^{n-1}(x)
\end{equation}
for every Borel set $\eta\subset S^{n-1}$.

This family encompasses several important geometric
measures. In particular,
\begin{align*}
\widetilde C_{p,q}(K,K,\cdot)
&=
\widetilde C_{p,n}(K,B,\cdot)
=
\frac1n S_p(K,\cdot),\\
\widetilde C_{0,q}(K,B,\cdot)
&=
\widetilde C_q(K,\cdot),\\
\widetilde C_{p,0}(K,B,\cdot)
&=
\frac1n J_p(K^*,\cdot).
\end{align*}
Thus, the $(p,q)$-th dual curvature measures connect several
important measures arising in the $L_p$ and dual
Brunn--Minkowski theories.

The $q$-th dual curvature measure $\widetilde C_q(K,\cdot)$
was introduced by Huang, Lutwak, Yang, and Zhang
\cite{HLYZ2016}. It arises as the variational derivative of
the dual volume and is given by
\begin{equation*}
\widetilde C_q(K,\eta)
=
\frac1n
\int_{n_K^{-1}(\eta)}
(x\cdot n_K(x))|x|^{q-n}
\,d\mathcal H^{n-1}(x)
\end{equation*}
for every Borel set $\eta\subset S^{n-1}$. In particular,
when $q=n$, it reduces to the cone volume measure
\[
\widetilde C_n(K,\cdot)
=
\frac1n S_0(K,\cdot).
\]
The $L_p$ integral curvature $J_p(K,\cdot)$, introduced by
Huang, Lutwak, Yang, and Zhang \cite{HLYZ2018}, can
equivalently be expressed as
\begin{equation*}
J_p(K,\eta)
=
\int_{n_{K^*}^{-1}(\eta)}
(x\cdot n_{K^*}(x))^{1-p}
|x|^{-n}
\,d\mathcal H^{n-1}(x)
\end{equation*}
for every Borel set $\eta\subset S^{n-1}$.

In recent years, considerable attention has been devoted to
extending problems from the classical Brunn--Minkowski theory
to the setting of general measures; see, for example,
\cite{AHRNZ,AH2023,CK2024,FLMZ1,FLMZ2,Hos1,HKL,KL2023,
LLT,LRZ2022,LP,LPRY2025,LLH2024,L2019,LM2017,MR2014,
RX2021,Wu2017,Wu2018,Zv2005}. This direction of research
is commonly known as the weighted Brunn--Minkowski theory.

Let $\mu$ be a Borel measure on $\mathbb R^n$ that is
absolutely continuous with respect to Lebesgue measure,
with a nonnegative locally integrable density $g$; that is,
\[
d\mu(x)=g(x)\,dx,
\qquad
g\in L^1_{\mathrm{loc}}(\mathbb R^n).
\]
For $K\in\mathcal K_o^n$, define
\begin{equation}\label{gK}
g_K(u):=g(\rho_K(u)u),
\qquad
\mathcal K_g^n
:=
\{K\in\mathcal K_o^n:g_K\in L^1(S^{n-1})\}.
\end{equation}
By Lemma~\ref{lem:gK-integrability}, the condition
$g_K\in L^1(S^{n-1})$ is equivalent to
$g\in L^1(\partial K,\mathcal H^{n-1})$.

We say that $K\in\mathcal K_o^n$ is
\emph{trace-regular with respect to $g$} if
\begin{equation}\label{radial-trace}
\lim_{\varepsilon\to0^+}
\frac1\varepsilon
\int_{S^{n-1}}
\int_{\rho_K(u)-\varepsilon}^{\rho_K(u)+\varepsilon}
|g(ru)-g(\rho_K(u)u)|
\,dr\,du
=0.
\end{equation}
This condition expresses averaged $L^1$-continuity of $g$
in the radial direction near $\partial K$. By
Lemma~\ref{trace-implies-integrability}, trace-regularity
implies that $K\in\mathcal K_g^n$. Moreover, condition
\eqref{radial-trace} provides the boundary regularity needed
to justify the variational formula in Theorem~\ref{main}.

Lemma~\ref{trace-examples} provides useful sufficient
conditions for trace-regularity. In particular, every
$K\in\mathcal K_o^n$ is trace-regular if $g$ is pointwise
$\alpha$-homogeneous for some $\alpha\in\mathbb R$.
The same conclusion holds if $g$ is radially non-increasing
and $r\mapsto g(ru)$ is continuous at $r=\rho_K(u)$ for
almost every $u\in S^{n-1}$. In particular, if $g$ is
continuous on $\mathbb R^n$, then every
$K\in\mathcal K_o^n$ is trace-regular. Both the Lebesgue
density, which is constant, and the Gaussian density satisfy
these conditions. Consequently,
$\mathcal K_g^n=\mathcal K_o^n$ for either measure.

For sufficiently regular densities, the following weighted
variational formula was established by Livshyts \cite{L2019}
for $p=1$ and by Wu \cite{Wu2017} for $p>1$. Under the
boundary condition \eqref{radial-trace}, Corollary~\ref{mu22}
below establishes the corresponding identity:
\begin{equation}\label{mixu2}
\lim_{t\to0}
\frac{\mu(K+_p t\cdot L)-\mu(K)}{t}
=
\int_{S^{n-1}}
h_L^{\bar p}(v)\,dS_p^\mu(K,v),
\end{equation}
where $S_p^\mu(K,\cdot)$ denotes the $L_p$ weighted surface
area measure, defined by
\begin{equation}\label{Smu}
S_p^\mu(K,\eta)
=
\int_{n_K^{-1}(\eta)}
(x\cdot n_K(x))^{1-p}
g(x)\,d\mathcal H^{n-1}(x)
\end{equation}
for every Borel set $\eta\subset S^{n-1}$. When $\mu$ is
Lebesgue measure $\lambda_n$ on $\mathbb R^n$, so that
$g\equiv1$, the measure $S_p^{\lambda_n}(K,\cdot)$ reduces
to the classical $L_p$ surface area measure $S_p(K,\cdot)$.
In this case, \eqref{mixu2} recovers the variational formula
\eqref{e-3}.

To connect $\mu$ with dual curvature measures, we introduce
the following quantity. For $K\in\mathcal K_o^n$ and
$Q,C\in\mathcal S_o^n$, define
\begin{align}\label{Vqmu}
\widetilde D_q^\mu(K,Q,C)
&:=
\frac1n\int_{K\setminus C}
\rho_Q^{n-q}(x)\,d\mu(x)
-\frac1n\int_{C\setminus K}
\rho_Q^{n-q}(x)\,d\mu(x)
\notag\\
&=
\frac1n\int_{S^{n-1}}
\int_{\rho_C(u)}^{\rho_K(u)}
\rho_Q^{n-q}(u)r^{q-1}g(ru)\,dr\,du.
\end{align}
The quantity $\widetilde D_q^\mu(K,Q,C)$ is well-defined.
Since $Q\in\mathcal S_o^n$, the function $\rho_Q$ is positive and continuous on $S^{n-1}$. Set
\[
a_Q:=\min_{u\in S^{n-1}}\rho_Q^{n-q}(u)>0,
\qquad
b_Q:=\max_{u\in S^{n-1}}\rho_Q^{n-q}(u)<\infty.
\]
By the $(-1)$-homogeneity of the radial function,
\begin{equation}\label{qn}
a_Q|x|^{q-n}
\le \rho_Q^{n-q}(x)
\le b_Q|x|^{q-n},
\qquad x\in\mathbb R^n\setminus\{o\}.
\end{equation}
It follows from \eqref{Vqmu} that
\begin{align*}
\big|\widetilde D_q^\mu(K,Q,C)\big|
&\le
\frac1n
\int_{(K\setminus C)\cup(C\setminus K)}
\rho_Q^{n-q}(x)\,d\mu(x)\\
&\le
\frac{b_Q}{n}
\int_{(K\setminus C)\cup(C\setminus K)}
|x|^{q-n}g(x)\,dx.
\end{align*}
Since $K$ and $C$ contain neighborhoods of the origin and
are bounded, their symmetric difference is contained in a compact
annulus separated from the origin. Hence $|x|^{q-n}$ is bounded on
this annulus, and the local integrability of $g$ guarantees that
the preceding integral is finite.

Let $p,q\in\mathbb R$, $K\in\mathcal K_g^n$,
$Q\in\mathcal S_o^n$, and let $\mu$ be a Borel measure on
$\mathbb R^n$ with density $g$. Motivated by the
$(p,q)$-th dual curvature measure \eqref{CqQ}, we introduce
the \emph{$(p,q)$-weighted dual curvature measure}
$\widetilde C_{p,q}^\mu(K,Q,\cdot)$, defined by
\begin{equation}\label{CWqQ}
\widetilde C_{p,q}^\mu(K,Q,\eta)
=
\frac1n
\int_{n_K^{-1}(\eta)}
(x\cdot n_K(x))^{1-p}
\rho_Q^{n-q}(x)g(x)
\,d\mathcal H^{n-1}(x)
\end{equation}
for every Borel set $\eta\subset S^{n-1}$.

Our first result is the following variational formula,
which extends \eqref{e1} to the setting of general measures.

\begin{thm}\label{main}
Let $p,q\in\mathbb R$, $K,L\in\mathcal K_o^n$,
$Q,C\in\mathcal S_o^n$, and let $\mu$ be a Borel measure
on $\mathbb R^n$ with density $g$. Assume that
$g\in L^1_{\mathrm{loc}}(\mathbb R^n)$ and that $K$
satisfies \eqref{radial-trace}. Then
\begin{equation}\label{Vv}
\lim_{t\to0}
\frac{
\widetilde D_q^\mu(K+_p t\cdot L,Q,C)
-\widetilde D_q^\mu(K,Q,C)
}{t}
=
\int_{S^{n-1}}
h_L^{\bar p}(v)\,
d\widetilde C_{p,q}^\mu(K,Q,v).
\end{equation}
\end{thm}
 When $\mu=\lambda_n$ and $q>0$, \eqref{VD} and
\eqref{lamv} show that the variational formula \eqref{Vv}
reduces to \eqref{e1}.

The Minkowski problem for $(p,q)$-th dual curvature
measures, introduced by Lutwak, Yang, and Zhang
\cite{LYZ2018}, is known as the $L_p$ dual Minkowski problem.
The case $p=0$ was originally proposed by Huang, Lutwak,
Yang, and Zhang \cite{HLYZ2016}. This problem unifies
several fundamental geometric problems, including the
$L_p$ Minkowski problem, the dual Minkowski problem, and
the $L_p$ Aleksandrov problem. More precisely, the $L_p$
dual Minkowski problem asks for necessary and sufficient
conditions under which a finite Borel measure on
$S^{n-1}$ can be represented as a $(p,q)$-th dual curvature
measure of a convex body, and for the extent to which the
body is uniquely determined. The related $L_p$ Minkowski,
dual Minkowski, and $L_p$ Aleksandrov problems remain only
partially solved, and important special cases, such as the
centro-affine and logarithmic Minkowski problems, remain
open. The $L_p$ dual Minkowski problem has also stimulated
substantial progress in the theory of fully nonlinear
elliptic partial differential equations. The existence,
uniqueness, and regularity of solutions have been extensively
studied; see, for example,
\cite{B2023,BF,CHZ,CL2018,CLYZ,FHLX,FLX2025,FYZ2025,
GWXY,GWXY2,HDZ,HLXZ,HLYZ2016,HXYZ2024,HZ2018,Ji,
LL2020,LSW2020,LYZ2018,LYZ2024,Zhao2017}.

Motivated by the $(p,q)$-weighted dual curvature measure,
we pose the following problem.

\medskip

\noindent\textbf{The $(p,q)$-weighted dual Minkowski problem.}
Let $p,q\in\mathbb R$, $Q\in\mathcal S_o^n$, and let $\nu$
be a finite Borel measure on $S^{n-1}$. Suppose that $\mu$
is a Borel measure on $\mathbb R^n$ with density $g$.
What conditions on $\mu$ and $\nu$ guarantee the existence
of a convex body $K\in\mathcal K_o^n$ satisfying
\[
d\nu(u)
=
c_{p,q}\,d\widetilde C_{p,q}^\mu(K,Q,u)
\]
for some constant $c_{p,q}>0$?

For a $C^2$ strictly positive support function $h$ such that
$\nabla_S^2h+hI$ is positive definite, set
\[
X(u)=\nabla_Sh(u)+h(u)u.
\]
If $d\nu=f\,du$ and $g(X(u))>0$, then, in the smooth
setting, the measure equation is equivalent to the
Monge-Ampère type equation
\begin{equation*}
\det(\nabla_S^2h+hI)
=
\frac{n}{c_{p,q}}\,
h^{p-1}\rho_Q(X)^{q-n}
\frac{f}{g(X)}.
\end{equation*}
Here $\nabla_S$ and $\nabla_S^2$ denote the spherical
gradient and Hessian, respectively. For nonsmooth bodies,
the problem is understood in the sense of measures.

To study the $(p,q)$-weighted dual Minkowski problem, let $\mu$
be a Borel measure on $\mathbb R^n$ with density
$g:\mathbb R^n\to[0,\infty)$. We assume that
\[
|x|^{q-n}g(x)\in L^1_{\mathrm{loc}}(\mathbb R^n).
\]
For $K,Q\in\mathcal S_o^n$, define the $q$-th weighted
dual mixed volume by
\begin{equation}\label{Vqmu1}
\widetilde V_q^\mu(K,Q)
:=
\frac1n\int_K\rho_Q^{n-q}(x)\,d\mu(x).
\end{equation}
We first verify that this quantity is finite. By \eqref{qn},
\[
\widetilde V_q^\mu(K,Q)
\le
\frac{b_Q}{n}
\int_K |x|^{q-n}g(x)\,dx
<\infty,
\]
where the last inequality follows from the local integrability
assumption and the boundedness of $K$. Thus,
$\widetilde V_q^\mu(K,Q)$ is well-defined. Moreover,
\eqref{Vqmu} and \eqref{Vqmu1} imply that
\begin{equation}\label{VD}
\widetilde D_q^\mu(K,Q,C)
=
\widetilde V_q^\mu(K,Q)
-
\widetilde V_q^\mu(C,Q).
\end{equation}

We next discuss two important special cases. First, let
$\mu=\lambda_n$ be Lebesgue measure on $\mathbb R^n$.
Since $|x|^{q-n}$ is locally integrable near the origin if
and only if $q>0$, polar coordinates and the definitions
\eqref{VKQ} and \eqref{Vqmu1} yield
\begin{equation}\label{lamv}
\widetilde V_q^{\lambda_n}(K,Q)
=
\widetilde V_{\bar q}(K,Q),
\qquad q>0.
\end{equation}
Second, when $q=n$, the local integrability condition reduces
to $g\in L^1_{\mathrm{loc}}(\mathbb R^n)$. Since
$\rho_Q^0(x)=1$ for $x\ne o$, the weighted dual mixed volume
simplifies to
\begin{equation}\label{muK}
\widetilde V_n^\mu(K,Q)=\frac1n\mu(K).
\end{equation}

Fix $p>0$, $q\in\mathbb R$, and $\beta>0$. Let
$\Lambda_{p,q,\beta}$ denote the class of measures with
an even density $g:\mathbb R^n\to[0,\infty)$ satisfying
\begin{equation}\label{admissibility}
\begin{gathered}
|x|^{q-n}g(x)\in L^1_{\mathrm{loc}}(\mathbb R^n),\\
\lim_{r\to\infty}\frac{A_q(r)^\beta}{r^p}=0,
\qquad
\lim_{r\to0^+}\frac{A_q(r)^\beta}{r^p}=\infty,\\
A_q(r)=\int_{rB}|x|^{q-n}g(x)\,dx.
\end{gathered}
\end{equation}
The class $\Lambda_{p,q,\beta}$ for $q=n$ was investigated by Kryvonos and Langharst \cite{KL2023}, 
who also formulated the corresponding even weighted Minkowski problem. 
The following theorem generalizes their existence result to all $q\in\mathbb{R}$ 
and establishes uniqueness for prescribed measures with finite support, 
assuming strict positivity of the density and imposing additional constraints on the parameters.

\begin{thm}[Existence and uniqueness]
\label{thm:exist-uniqueness}
Let $p>0$, $q\in\mathbb R$, $\beta>0$, and
$Q\in\mathcal S_e^n$. Suppose that
$\mu\in\Lambda_{p,q,\beta}$ has a strictly positive density
$g$, and that every $K\in\mathcal K_o^n$ satisfies
\eqref{radial-trace}. Let $\nu$ be a finite even Borel
measure on $S^{n-1}$ that is not concentrated on any
great subsphere. Then there exists
$K_0\in\mathcal K_e^n$ satisfying
\begin{equation*}
d\nu
=
\frac1p
\widetilde V_q^\mu(K_0,Q)^{\beta-1}
\,d\widetilde C_{p,q}^\mu(K_0,Q,\cdot).
\end{equation*}
Furthermore, suppose that $\nu$ has finite support and
$\beta\ge1$. If, in addition, either of the following
conditions holds:
\begin{enumerate}[label=\textup{(\roman*)}]
\item $g$ is pointwise $\alpha$-homogeneous for some
$\alpha>-q$, and $\beta(q+\alpha)<p$;
\item $q>0$, $g$ is radially non-increasing and continuous,
and $\beta q<p$,
\end{enumerate}
then the solution is unique in the class of
origin-symmetric convex bodies.
\end{thm}

In both cases (i) and (ii) of
Theorem~\ref{thm:exist-uniqueness},
Lemma~\ref{trace-examples} and
Theorem~\ref{cor:examples-Lambda} guarantee, respectively,
that every $K\in\mathcal K_o^n$ satisfies
\eqref{radial-trace} and that $\mu\in\Lambda_{p,q,\beta}$,
provided that
$g\in L^1_{\mathrm{loc}}(\mathbb R^n)$ and
$
\lambda_n(\{x\in\mathbb R^n:g(x)>0\})>0
$.
Consequently, Theorem~\ref{thm:exist-uniqueness} applies to
Lebesgue and Gaussian measures whenever $\nu$ is a finite
even Borel measure on $S^{n-1}$ that is not concentrated on
any great subsphere. For related results on the Gaussian
dual Minkowski problem, see \cite{FHLX,FLX2025}.

The $(p,q)$-weighted dual curvature measure also gives rise
to a family of weighted dual mixed volumes. For
$p,q\in\mathbb R$, $K\in\mathcal K_g^n$,
$L\in\mathcal K_o^n$, and $Q\in\mathcal S_o^n$, define
the $(p,q)$-weighted dual mixed volume by
\begin{equation}\label{Vpq}
\widetilde V_{p,q}^\mu(K,L,Q)
:=
\int_{S^{n-1}}
h_L^p(v)\,d\widetilde C_{p,q}^\mu(K,Q,v).
\end{equation}
This definition extends the classical $(p,q)$-dual mixed
volume introduced by Lutwak, Yang, and Zhang \cite{LYZ2018},
which it recovers when $\mu$ is Lebesgue measure.
The affine properties of the $(p,q)$-weighted dual
curvature measure and its associated mixed volume are
investigated in Theorems~\ref{pp1} and~\ref{slin}.

The paper is organized as follows. Section~\ref{2}
presents preliminary material. Section~\ref{3} establishes
properties of $(p,q)$-weighted dual curvature measures,
while Section~\ref{4} addresses trace-regularity.
Sections~\ref{5} and~\ref{6} contain the proofs of
Theorems~\ref{main} and~\ref{thm:exist-uniqueness},
respectively. Section~\ref{7} discusses the affine
properties of $(p,q)$-weighted dual curvature measures
and their associated dual mixed volumes.

The arguments below draw on ideas from Kryvonos and
Langharst \cite{KL2023} and Lutwak, Yang, and Zhang
\cite{LYZ2018}, particularly concerning variational
formulas and the associated Minkowski problem.

\section{Preliminaries}\label{2}

In this section, we recall some basic facts and notation concerning convex bodies that will be used throughout the paper. Comprehensive references for these materials can be found in the monographs \cite{Gardner,Schneider2014}.

Let \(C(S^{n-1})\) denote the space of continuous functions on the unit sphere \(S^{n-1}\), endowed with the topology induced by the maximum norm
\[
\|f\|_\infty=\max_{v\in S^{n-1}}|f(v)|,
\]
and let \(C^+(S^{n-1})\) be the subset of \(C(S^{n-1})\) consisting of strictly positive functions. For each \(f\in C^+(S^{n-1})\), the Wulff shape generated by \(f\) is the convex body
\begin{equation}\label{Wulff}
[f]=\bigl\{x\in\mathbb{R}^n:x\cdot v\leq f(v)\ \text{for all }v\in S^{n-1}\bigr\}.
\end{equation}
It is well known and follows immediately from the definition that
\begin{equation*}
    h_{[f]} \leq f \qquad \text{and} \qquad [h_K] = K
\end{equation*}
for every $K \in \mathcal{K}_o^n$.

For $K\in\mathcal K_o^n$, the support and radial functions are
homogeneous of degrees $1$ and $-1$, respectively. That is, for
$c>0$,
\begin{equation}\label{hhom}
h_K(cx)=c\,h_K(x),\qquad x\in\mathbb R^n,
\end{equation}
and
\begin{equation}\label{hom}
\rho_K(cx)=c^{-1}\rho_K(x),\qquad
x\in\mathbb R^n\setminus\{o\}.
\end{equation}

For $\phi\in\operatorname{GL}(n)$, let
$\phi K=\{\phi x:x\in K\}$. Then
\begin{equation}\label{hphi}
h_{\phi K}(x)=h_K(\phi^t x),\qquad x\in\mathbb R^n,
\end{equation}
and
\begin{equation}\label{phin}
\rho_{\phi K}(x)=\rho_K(\phi^{-1}x),\qquad
x\in\mathbb R^n\setminus\{o\}.
\end{equation}

For the polar body \(K^*\), we have
\[
h_{K^*}(x)=\frac{1}{\rho_K(x)}
\qquad\text{and}\qquad
\rho_{K^*}(x)=\frac{1}{h_K(x)},
\qquad x\in\mathbb R^n\setminus\{o\}.
\]

Finally, for a sequence of convex bodies \(K_i\in\mathcal K_o^n\), we say that \(K_i\) converges to \(K_0\in\mathcal K_o^n\) in the Hausdorff metric if
\[
\|h_{K_i}-h_{K_0}\|_\infty
=
\max_{v\in S^{n-1}}
|h_{K_i}(v)-h_{K_0}(v)|
\longrightarrow0
\]
as \(i\to\infty\).

\subsection{The radial Gauss map}

The supporting hyperplane of $K\in\mathcal{K}_o^n$ for each $v\in S^{n-1}$ is defined by
\begin{equation*}
H_K(v)=\{x\in\mathbb{R}^n:\langle x,v\rangle=h_K(v)\}.
\end{equation*}
For a set $\sigma\subset\partial K$, the spherical image of $\sigma$ is defined by
\begin{equation*}
\pmb{n}_K(\sigma)=\bigl\{v\in S^{n-1}:x\in H_K(v)\text{ for some }x\in\sigma\bigr\}\subset S^{n-1}.
\end{equation*}
Let \(\sigma_K\subset\partial K\) denote the set of all points \(x\in\partial K\) for which the set $\pmb{n}_K(\{x\})$, often abbreviated as $\pmb{n}_K(x)$, is not a singleton. It is well known that \(\mathcal H^{n-1}(\sigma_K)=0\) (see \cite[ p.~84]{Schneider2014}).  The spherical image map (also called the Gauss map) of \(K\) is the function
\begin{equation*}
n_K : \partial K\setminus\sigma_K \to  S^{n-1},
\end{equation*}
by letting $n_K(x)$ be the unique element in $\pmb{n}_K(x)$ for each $x\in\partial K\setminus\sigma_K$. This map is known to be continuous (see \cite[Lemma 2.2.12]{Schneider2014}). For brevity, we shall write \(\partial' K := \partial K\setminus\sigma_K\).

For \(K\in\mathcal K_o^n\), the radial map \(r_K: S^{n-1}\to\partial K\) is defined by
\begin{equation*}
r_K(u)=\rho_K(u)u,\qquad u\in S^{n-1}.
\end{equation*}
For a set \(\omega\subset S^{n-1}\), the radial Gauss image of \(K\) at \(\omega\) is the set
\begin{equation}\label{ala}
\pmb\alpha_K(\omega)=\bigl\{v\in S^{n-1}: r_K(u)\in H_K(v)\text{ for some }u\in\omega\bigr\}.
\end{equation}
When \(\omega=\{u\}\), we write \(\pmb\alpha_K(u)\) for \(\pmb\alpha_K(\{u\})\).
Let
\begin{equation}\label{ome}
	\omega_K=\{u\in S^{n-1}: r_K(u)\in\sigma_K\}.
\end{equation}
For each \(u\in\omega_K\), the set \(\pmb\alpha_K(u)\) contains more than one point. It is known that \(\mathcal H^{n-1}(\omega_K)=0\) (see \cite[Theorem 2.2.5]{Schneider2014}). For \(u\in S^{n-1}\setminus\omega_K\), \(\pmb\alpha_K(u)\) is a singleton.  This yields the radial Gauss map
\begin{equation*}
\alpha_K: S^{n-1}\setminus\omega_K\to S^{n-1},
\end{equation*}
by letting $\alpha_K(u)$ be the unique element in \(\pmb\alpha_K(u)\) for each $u\in S^{n-1}\setminus\omega_K$.
Hence $\alpha_K$ is defined almost everywhere on $ S^{n-1}$ with respect to spherical Lebesgue measure.

For a set $\eta\subset S^{n-1}$, the reverse radial Gauss image $\pmb\alpha_K^*(\eta)$ of $K\in\mathcal{K}_o^n$ is defined by
\begin{equation}\label{sta}
\pmb\alpha_K^*(\eta)=\bigl\{u\in S^{n-1}: r_K(u)\in H_K(v)\text{ for some }v\in\eta\bigr\}\subset S^{n-1}.
\end{equation}
We write $\pmb\alpha_K^*(v)$ instead of $\pmb\alpha_K^*(\{v\})$ when $\eta=\{v\}$. Suppose that $\eta_K\subset S^{n-1}$ consists of all $v\in S^{n-1}$ such that $\pmb\alpha_K^*(v)$ contains more than a single element. Obviously, for $v\in S^{n-1}\setminus\eta_K$, $\pmb\alpha_K^*(v)$ has a unique element. Therefore,  the reverse radial Gauss map is defined by
\begin{equation*}
\alpha_K^*: S^{n-1}\setminus\eta_K\to S^{n-1},
\end{equation*}
by letting $\alpha_K^*(v)$ be the unique element in \(\pmb\alpha_K^*(v)\) for each $v\in S^{n-1}\setminus\eta_K$. 
It was shown in \cite{HLYZ2016} that for $K\in\mathcal{K}_o^n$,
\begin{equation*}
\alpha_K^*=\alpha_{K^*}
\end{equation*}
almost everywhere on $ S^{n-1}$ with respect to spherical Lebesgue measure.

The following integral identity was established in
\cite[Lemma 2.9]{HLYZ2016}: if $q\in\mathbb R$, $K\in\mathcal
{K}_{o}^n$, and $\Phi:S^{n-1}\to \mathbb R$ is bounded and
Lebesgue integrable, then
\begin{equation}\label{LYZ20}\int_{S^{n-1}}\Phi(u)\rho_K(u)^qdu=\int_{\partial^{\prime}K}\Phi(\overline{x})|x|^{q-n}(x\cdot n_K(x))d\mathcal{H}^{n-1}(x),
\end{equation}
where $\overline{x}=x/|x|$.

\subsection{The $L_p$ Minkowski-Firey combination and the $L_p$ radial combination}

Building upon the Wulff shape construction \eqref{Wulff}, the definition of the $L_p$ Minkowski-Firey combination \eqref{lp} can be extended to all $p\in\mathbb{R}$, as well as to negative coefficients $s$ and $t$. For $p \neq 0$, $K,L \in \mathcal{K}_o^n$, and $s,t \in \mathbb{R}$ such that $s h_K^p + t h_L^p$ is strictly positive on $S^{n-1}$,  the $L_p$ Minkowski-Firey combination $s\cdot K +_p t\cdot L \in \mathcal{K}_o^n$ is defined by
\begin{equation}\label{eq:2.11}
s\cdot K +_p t\cdot L := \bigl[(s h_K^p + t h_L^p)^{1/p}\bigr]=\bigcap_{v\in S^{n-1}} \bigl\{x\in\mathbb{R}^n \colon x\cdot v \leq \bigl(s h_{K}(v)^p + t h_{L}(v)^p\bigr)^{1/p}\bigr\}.
\end{equation}
For $p=0$, set
\begin{equation}\label{eq:2.12}
s\cdot K +_0 t\cdot L := \bigl[h_K^s h_L^t\bigr]=\bigcap_{v\in S^{n-1}} \bigl\{x\in\mathbb{R}^n \colon x\cdot v \leq h_{K}(v)^s h_{L}(v)^t\bigr\}.
\end{equation}
Since $h_K$ and $h_L$ are strictly positive on $S^{n-1}$, the combination $s\cdot K +_0 t\cdot L$ is well defined for all $s,t\in\mathbb{R}$.

One readily verifies that for $\phi\in\mathrm{SL}(n)$ and $p\neq0$,
\begin{equation}\label{aph}
	s\cdot \phi K +_p t\cdot \phi L = \phi\bigl(s\cdot K +_p t\cdot L\bigr).
\end{equation}
If additionally $s +t =1$, then \eqref{aph} holds for $p=0$ as well. Indeed, this follows directly from the Wulff-shape property $[f(\phi^t\cdot)] = \phi[f]$ for positively homogeneous functions $f$ of degree $1$, the definition of the $L_p$ Minkowski--Firey combination, and the identity $h_{\phi K}(v)=h_K(\phi^t v)$.

For $K,L\in\mathcal S_o^n$ and $q\ne0$, the $L_q$ radial combination is defined by
\[
\rho_{s\cdot K\widetilde+_q t\cdot L}(u)
=\bigl(s\rho_K(u)^q+t\rho_L(u)^q\bigr)^{1/q},\qquad u\in S^{n-1},
\]
provided the expression inside the parentheses is strictly positive. For $q=0$, set
\[
\rho_{s\cdot K\widetilde+_0 t\cdot L}(u)=\rho_K(u)^s\rho_L(u)^t,
\qquad u\in S^{n-1},
\]
and extend from the sphere with degree $-1$. For $\phi\in\mathrm{SL}(n)$,
\begin{equation*}
s\cdot\phi K\widetilde+_q t\cdot\phi L
=\phi(s\cdot K\widetilde+_q t\cdot L)
\end{equation*}
holds for $q\ne0$, and also for $q=0$ when $s+t=1$.
 
For $q\in\mathbb R\setminus\{0\}$ and $K,L\in\mathcal S_o^n$, the $q$-th dual mixed volume $\widetilde V_{q}(K,L)$ is defined by
\begin{equation}\label{wV}
	\widetilde{V}_{q}(K,L) 
	:= \frac{q}{n}\lim_{t\to 0}\frac{V(L \widetilde{+}_q t\cdot K)-V(L)}{t}=\frac{1}{n}\int_{S^{n-1}} \rho_K^{q}(u)\rho_L^{n-q}(u)\,du,
\end{equation}

\section{Properties of $(p,q)$-weighted dual curvature measures}\label{3}

The following proposition can be proved by adapting the arguments in \cite{LYZ2018}. For completeness, we provide a proof with the boundary integrability assumption made explicit.

\begin{prop}\label{representation}
Let $p,q\in\mathbb R$, $K\in\mathcal K_g^n$, and $Q\in\mathcal S_o^n$. Then for each bounded Borel function $f : S^{n-1} \to \mathbb{R}$,
\begin{equation}\label{12}
\int_{S^{n-1}}f(v)\,d\widetilde{C}^{\mu}_{p,q}(K,Q,v)=\frac{1}{n} \int_{\partial^{\prime}K} f(n_K(x)) (x\cdot n_K(x))^{1-p}\rho^{n-q}_Q(x)\,g(x)\,d\mathcal{H}^{n-1}(x),
\end{equation}
\begin{equation}\label{23}
\int_{S^{n-1}}f(v)\,d\widetilde{C}^{\mu}_{p,q}(K,Q,v)=\frac{1}{n} \int_{S^{n-1}} f(\alpha_K(u))\,h_K^{-p}(\alpha_K(u))\,\rho_K^{q}(u)\,\rho_Q^{n-q}(u)\,g(r_K(u))\,du. 
\end{equation}
\end{prop}

\begin{proof}
Let $\phi$ be a simple function on $S^{n-1}$ given by
\begin{equation*}
\phi = \sum_{i=1}^{m} c_i \mathbf{1}_{\eta_i},
\end{equation*}
where $c_i \in \mathbb{R}$ and the $\eta_i$ are Borel subsets of $S^{n-1}$. It follows from \eqref{CWqQ} that
\begin{align}
&\int_{S^{n-1}} \phi(v)\,d\widetilde{C}^{\mu}_{p,q}(K,Q,v)
\notag\\&= \int_{S^{n-1}} \sum_{i=1}^{m} c_i \mathbf{1}_{\eta_i}(v)\,d\widetilde{C}^{\mu}_{p,q}(K,Q,v)\notag \\
&= \sum_{i=1}^{m} c_i \widetilde{C}^{\mu}_{p,q}(K,Q,\eta_i) \notag\\
&= \sum_{i=1}^{m} \frac{c_i}{n}\int_{n^{-1}_K (\eta_i)}(x\cdot n_K(x))^{1-p}\rho^{n-q}_Q(x)\,g(x)\,d\mathcal{H}^{n-1}(x)\notag\\
&= \sum_{i=1}^{m} \frac{c_i}{n}\int_{\partial^{\prime}K} \mathbf{1}_{\eta_i}(n_K(x))(x\cdot n_K(x))^{1-p}\rho^{n-q}_Q(x)\,g(x)\,d\mathcal{H}^{n-1}(x)\notag \\
&=\frac{1}{n} \int_{\partial^{\prime}K} \phi(n_K(x)) (x\cdot n_K(x))^{1-p}\rho^{n-q}_Q(x)\,g(x)\,d\mathcal{H}^{n-1}(x).\label{alp}
\end{align}

Taking $\Phi = (\phi\circ \alpha_{K})(h_{K}^{-p}\circ \alpha_{K}) (g\circ r_{K})\rho_{Q}^{n-q}$ in \eqref{LYZ20} and using the fact that $\rho_{Q}$ is homogeneous of degree $-1$, we further obtain
\begin{align}
&\int_{S^{n-1}} \phi(v)\,d\widetilde{C}^{\mu}_{p,q}(K,Q,v)
\notag\\
&=\frac{1}{n}\int_{\partial^{\prime}K} \phi(\alpha_K(\overline{x}))\,h_K^{-p}(\alpha_K(\overline{x}))\,\rho_Q^{n-q}(\overline{x})\,g(r_K(\overline{x}))\,|x|^{q-n} (x\cdot n_K(x))\,d\mathcal{H}^{n-1}(x)\notag
\\
&=\frac{1}{n} \int_{S^{n-1}} \phi(\alpha_K(u))\,h_K^{-p}(\alpha_K(u))\,\rho_K^{q}(u)\,\rho_Q^{n-q}(u)\,g(r_K(u))\,du.\label{alp2}
\end{align}

From \eqref{alp} and \eqref{alp2}, we conclude that \eqref{12} and \eqref{23} hold for simple functions.
Now, for any bounded Borel function $f$, there exists a sequence of simple functions $\{\phi_k\}_{k=1}^{\infty}$
such that
\[
\|\phi_k-f\|_{\infty}\rightarrow 0 .
\]
Since $\phi_k\rightarrow f$ uniformly, we have
\[
\phi_k(\alpha_K(u))\rightarrow f(\alpha_K(u))
\]
for every $u\in S^{n-1}\setminus\omega_K$.
Moreover, since the radial Gauss map $\alpha_K$ is continuous on
$S^{n-1}\setminus\omega_K$, where $\omega_K$ is defined in \eqref{ome},
the function $f\circ\alpha_K$ is Borel measurable on
$S^{n-1}\setminus\omega_K$. Since $\omega_K$ has spherical Lebesgue
measure zero and $f$ is bounded, the function $f\circ\alpha_K$ is
Lebesgue integrable on $S^{n-1}$.

To pass to the limit in \eqref{alp2}, observe that
$g_K(u)=g(r_K(u))\in L^1(S^{n-1})$ by the assumption
$K\in\mathcal K_g^n$. Furthermore, since $K,Q\in\mathcal K_o^n$ and
$\mathcal S_o^n$, respectively, the functions
$
h_K^{-p}(\alpha_K(u))$,$
\rho_K^q(u)$,
$\rho_Q^{\,n-q}(u)$
are bounded on $S^{n-1}$. Hence, there exists a constant $C>0$ such that
$$
\left|
\phi_k(\alpha_K(u))
h_K^{-p}(\alpha_K(u))
\rho_K^q(u)
\rho_Q^{n-q}(u)
g(r_K(u))
\right|\leq
C\|f\|_{\infty}g_K(u),
$$
for all $k$ and almost every $u\in S^{n-1}$.
Since $g_K\in L^1(S^{n-1})$, the right-hand side is integrable.
Therefore, by the dominated convergence theorem,
\[
\begin{aligned}
&\lim_{k\to\infty}
\int_{S^{n-1}}
\phi_k(\alpha_K(u))
h_K^{-p}(\alpha_K(u))
\rho_K^q(u)
\rho_Q^{n-q}(u)
g(r_K(u))\,du\\
&=
\int_{S^{n-1}}
f(\alpha_K(u))
h_K^{-p}(\alpha_K(u))
\rho_K^q(u)
\rho_Q^{n-q}(u)
g(r_K(u))\,du .
\end{aligned}
\]
Similarly, by the boundedness of $f$ and the finiteness of
$\widetilde C^\mu_{p,q}(K,Q,\cdot)$,
\[
\lim_{k\to\infty}
\int_{S^{n-1}}\phi_k(v)\,
d\widetilde C^\mu_{p,q}(K,Q,v)
=
\int_{S^{n-1}}f(v)\,
d\widetilde C^\mu_{p,q}(K,Q,v).
\]
Passing to the limit in the identity for simple functions gives
\eqref{23}. Combining this with \eqref{alp}, we conclude that
\eqref{12} and \eqref{23} hold for every bounded Borel function $f$.

\end{proof}

Letting $f = \mathbf{1}_{\eta}$ in \eqref{23} and \eqref{sta}, we may rewrite \eqref{CWqQ} as
\begin{align} \label{all}
\widetilde{C}^{\mu}_{p, q}(K, Q, \eta) 
&= \int_{S^{n-1}} \mathbf{1}_{\eta}(u) \, d\widetilde{C}^{\mu}_{p, q}(K, Q, u) \notag\\
&= \frac{1}{n} \int_{S^{n-1}} \mathbf{1}_{\eta}\bigl(\alpha_{K}(u)\bigr) \, h_{K}^{-p}\bigl(\alpha_{K}(u)\bigr) \, \rho_{K}^{q}(u) \, \rho_{Q}^{n-q}(u) \, g(r_K(u)) \, du \notag \\
&= \frac{1}{n} \int_{S^{n-1}} \mathbf{1}_{\pmb\alpha_{K}^{*}(\eta)}(u) \, h_{K}^{-p}\bigl(\alpha_{K}(u)\bigr) \, \rho_{K}^{q}(u) \, \rho_{Q}^{n-q}(u) \, g(r_K(u)) \, du \notag \\
&= \frac{1}{n} \int_{\pmb\alpha_{K}^{*}(\eta)} h_{K}^{-p}\bigl(\alpha_{K}(u)\bigr) \, \rho_{K}^{q}(u) \, \rho_{Q}^{n-q}(u) \, g(r_K(u)) \, du.
\end{align}
Furthermore, by \eqref{CWqQ}, we have
\begin{equation*}
\widetilde{C}^{\mu}_{p,n}(K,B,\eta) = \frac{1}{n}\int_{n_K^{-1}(\eta)}\bigl(x\cdot n_K(x)\bigr)^{1-p} g(x) \, d\mathcal{H}^{n-1}(x) = \widetilde{C}^{\mu}_{p,q}(K, K,\eta)
\end{equation*}
for every Borel set $\eta \subseteq S^{n-1}$. Combining this identity with \eqref{Smu} immediately yields the following identities: for $K\in\mathcal K_g^n$,
\begin{equation}\label{s1}
\widetilde{C}^{\mu}_{p,q}(K, K, \cdot) = \widetilde{C}^{\mu}_{p,n}(K, B, \cdot) = \frac{1}{n} S^{\mu}_p(K, \cdot).
\end{equation}

\begin{prop}\label{pro}
Let $p,q\in\mathbb R$ and $Q\in\mathcal S_o^n$.
Assume that $g\in L^1_{\mathrm{loc}}(\mathbb R^n)$.
Then the following assertions hold:
\begin{enumerate}[label=\textup{(\roman*)}]
\item If $K_i,K\in\mathcal K_g^n$, $K_i\to K$, and
\[
\|g_{K_i}-g_K\|_{L^1(S^{n-1})}\to0,
\]
then $\widetilde C_{p,q}^\mu(K_i,Q,\cdot)$ converges weakly
to $\widetilde C_{p,q}^\mu(K,Q,\cdot)$.

\item For every $K\in\mathcal K_g^n$, the measure
$\widetilde C_{p,q}^\mu(K,Q,\cdot)$ is absolutely continuous
with respect to $S_1^\mu(K,\cdot)$.

\item If $g$ is continuous on $\mathbb R^n$, then for
$K,L\in\mathcal K_g^n$ such that $K\cup L$ is convex,
we have $K\cup L,K\cap L\in\mathcal K_g^n$ and
\[
\begin{aligned}
\widetilde C_{p,q}^\mu(K\cup L,Q,\cdot)
+\widetilde C_{p,q}^\mu(K\cap L,Q,\cdot)
=
\widetilde C_{p,q}^\mu(K,Q,\cdot)
+\widetilde C_{p,q}^\mu(L,Q,\cdot).
\end{aligned}
\]
\end{enumerate}
\end{prop}

\begin{proof}
We prove the three assertions separately.

\medskip
\noindent
\textup{(i)}
Let $f\in C(S^{n-1})$. By the integral representation
following from \eqref{23}, we have
\[
\int_{S^{n-1}} f(v)\,
d\widetilde C_{p,q}^\mu(K_i,Q,v)
=
\int_{S^{n-1}}\varphi_i(u)g_{K_i}(u)\,du,
\]
where
\[
\varphi_i(u)
=
\frac1n f(\alpha_{K_i}(u))
h_{K_i}^{-p}(\alpha_{K_i}(u))
\rho_{K_i}^q(u)\rho_Q^{n-q}(u).
\]
Let $\varphi$ denote the corresponding function associated
with $K$.

Since $K_i\to K$ in the Hausdorff metric and
$o\in\operatorname{int}K$, the support and radial functions
of $K_i$ are uniformly bounded above and bounded away from
zero for all sufficiently large $i$. Consequently, there
exists $M>0$ such that
\begin{equation}\label{MM}
|\varphi_i(u)|\le M,\qquad |\varphi(u)|\le M
\end{equation}
for all $u\in S^{n-1}$ and all sufficiently large $i$.

Moreover, $h_{K_i}\to h_K$ and $\rho_{K_i}\to\rho_K$
uniformly on $S^{n-1}$. By
\cite[Lemma~3.5]{LYZ2018},
$\alpha_{K_i}\to\alpha_K$ almost everywhere on $S^{n-1}$.
It follows that
\[
\varphi_i(u)\to\varphi(u)
\quad\text{for almost every }u\in S^{n-1}.
\]

By the triangle inequality,
\begin{align*}
&\left|
\int_{S^{n-1}}\varphi_i(u)g_{K_i}(u)\,du
-\int_{S^{n-1}}\varphi(u)g_K(u)\,du
\right|\\
&\quad\le
\int_{S^{n-1}}|\varphi_i(u)|
|g_{K_i}(u)-g_K(u)|\,du+
\int_{S^{n-1}}|\varphi_i(u)-\varphi(u)|
g_K(u)\,du.
\end{align*}
The first term is bounded by
$M\|g_{K_i}-g_K\|_{L^1(S^{n-1})}$, which tends to zero
by assumption. For the second term, \eqref{MM} gives
\[
|\varphi_i(u)-\varphi(u)|g_K(u)\le 2M g_K(u).
\]
Since $g_K\in L^1(S^{n-1})$, the dominated convergence
theorem implies that the second term also tends to zero.
Therefore,
\[
\lim_{i\to\infty}
\int_{S^{n-1}}f(v)\,
d\widetilde C_{p,q}^\mu(K_i,Q,v)
=
\int_{S^{n-1}}f(v)\,
d\widetilde C_{p,q}^\mu(K,Q,v).
\]
Since $f\in C(S^{n-1})$ is arbitrary, we conclude that
\[
\widetilde C_{p,q}^\mu(K_i,Q,\cdot)
\stackrel{w}{\longrightarrow}
\widetilde C_{p,q}^\mu(K,Q,\cdot).
\]

\medskip
\noindent
\textup{(ii)}
Let $\eta\subset S^{n-1}$ be a Borel set such that
\[
S_1^\mu(K,\eta)=0.
\]
By the definition of $S_1^\mu(K,\cdot)$, this means that
\[
\int_{n_K^{-1}(\eta)}
g(x)\,d\mathcal H^{n-1}(x)=0.
\]
On the other hand, by \eqref{12},
\[
\widetilde C_{p,q}^\mu(K,Q,\eta)
=
\frac1n
\int_{n_K^{-1}(\eta)}
(x\cdot n_K(x))^{1-p}
\rho_Q^{n-q}(x)g(x)\,
d\mathcal H^{n-1}(x).
\]
The weight
$(x\cdot n_K(x))^{1-p}\rho_Q^{n-q}(x)$
is finite on the relevant part of $\partial K$.
Hence the integral vanishes, and therefore
\[
\widetilde C_{p,q}^\mu(K,Q,\eta)=0.
\]

\medskip
\noindent
\textup{(iii)}
We first note that
\[
\rho_{K\cup L}(u)=\max\{\rho_K(u),\rho_L(u)\},
\qquad
\rho_{K\cap L}(u)=\min\{\rho_K(u),\rho_L(u)\}.
\]
It follows from \eqref{gK} that
\[
g_{K\cup L}(u)+g_{K\cap L}(u)=g_K(u)+g_L(u).
\]
Since $g_K,g_L\in L^1(S^{n-1})$, we obtain
$K\cup L,K\cap L\in\mathcal K_g^n$.

It remains to prove the valuation identity. We first assume
that $K$ and $L$ are strictly convex. Let
$\eta\subset S^{n-1}$ be a Borel set. Decompose $\eta$ into
three pairwise disjoint Borel sets,
\[
\eta=\eta_0\cup\eta_K\cup\eta_L,
\]
where
\begin{align*}
\eta_K&=\{v\in\eta:h_K(v)>h_L(v)\},\\
\eta_L&=\{v\in\eta:h_K(v)<h_L(v)\},\\
\eta_0&=\{v\in\eta:h_K(v)=h_L(v)\}.
\end{align*}
By the boundary decomposition, the corresponding boundary
pieces agree and have the same outer unit normals. Thus, by
\eqref{CWqQ}, we obtain
\[
\begin{aligned}
\int_{n_{K\cup L}^{-1}(\eta_K)}
(x\cdot n_{K\cup L}(x))^{1-p}
\rho_Q^{n-q}(x)g(x)\,
d\mathcal H^{n-1}(x)=
\int_{n_K^{-1}(\eta_K)}
(x\cdot n_K(x))^{1-p}
\rho_Q^{n-q}(x)g(x)\,
d\mathcal H^{n-1}(x),
\end{aligned}
\]
and
\[
\begin{aligned}
\int_{n_{K\cap L}^{-1}(\eta_K)}
(x\cdot n_{K\cap L}(x))^{1-p}
\rho_Q^{n-q}(x)g(x)\,
d\mathcal H^{n-1}(x)=
\int_{n_L^{-1}(\eta_K)}
(x\cdot n_L(x))^{1-p}
\rho_Q^{n-q}(x)g(x)\,
d\mathcal H^{n-1}(x).
\end{aligned}
\]
Therefore,
\[
\begin{aligned}
\widetilde C_{p,q}^\mu(K\cup L,Q,\eta_K)
&=\widetilde C_{p,q}^\mu(K,Q,\eta_K),\\
\widetilde C_{p,q}^\mu(K\cap L,Q,\eta_K)
&=\widetilde C_{p,q}^\mu(L,Q,\eta_K).
\end{aligned}
\]
Similarly, on $\eta_L$,
\[
\begin{aligned}
\widetilde C_{p,q}^\mu(K\cup L,Q,\eta_L)
&=\widetilde C_{p,q}^\mu(L,Q,\eta_L),\\
\widetilde C_{p,q}^\mu(K\cap L,Q,\eta_L)
&=\widetilde C_{p,q}^\mu(K,Q,\eta_L).
\end{aligned}
\]
On $\eta_0$, the corresponding boundary decomposition gives
\[
\begin{aligned}
\widetilde C_{p,q}^\mu(K\cup L,Q,\eta_0)
&=\widetilde C_{p,q}^\mu(K,Q,\eta_0),\\
\widetilde C_{p,q}^\mu(K\cap L,Q,\eta_0)
&=\widetilde C_{p,q}^\mu(L,Q,\eta_0).
\end{aligned}
\]
Adding these identities and using the additivity of
$\widetilde C_{p,q}^\mu(M,Q,\cdot)$ as a measure, we obtain
\[
\begin{aligned}
\widetilde C_{p,q}^\mu(K\cup L,Q,\eta)
+\widetilde C_{p,q}^\mu(K\cap L,Q,\eta)=
\widetilde C_{p,q}^\mu(K,Q,\eta)
+\widetilde C_{p,q}^\mu(L,Q,\eta).
\end{aligned}
\]
Since $\eta$ is arbitrary, the valuation identity holds
for strictly convex $K$ and $L$.

For arbitrary $K,L\in\mathcal K_g^n$ such that $K\cup L$
is convex, Weil's Approximation Lemma
\cite[p.~118]{LYZ2018} yields approximating sequences
of strictly convex bodies $K_i,L_i\in\mathcal K_o^n$ such that
\[
K_i\to K,\qquad L_i\to L,\qquad
K_i\cup L_i\to K\cup L,\qquad
K_i\cap L_i\to K\cap L,
\]
with $K_i\cup L_i$ convex for every $i$.
Since $g$ is continuous on $\mathbb R^n$, Hausdorff convergence
of these bodies implies convergence of the corresponding
functions in $L^1(S^{n-1})$. By part~\textup{(i)}, the associated
curvature measures converge weakly. Passing to the limit in
the valuation identity for $K_i$ and $L_i$, we obtain
\[
\begin{aligned}
\widetilde C_{p,q}^\mu(K\cup L,Q,\cdot)
+\widetilde C_{p,q}^\mu(K\cap L,Q,\cdot)=
\widetilde C_{p,q}^\mu(K,Q,\cdot)
+\widetilde C_{p,q}^\mu(L,Q,\cdot).
\end{aligned}
\]
This completes the proof.

\end{proof}

\section{Trace-regularity condition}\label{4}

Recall from \eqref{radial-trace} that $K\in\mathcal K_o^n$ is
trace-regular with respect to $g$ if
\begin{equation}\label{var}
\lim_{\varepsilon\to0^+}
\frac1\varepsilon
\int_{S^{n-1}}
\int_{\rho_K(u)-\varepsilon}^{\rho_K(u)+\varepsilon}
|g(ru)-g_K(u)|
\,dr\,du
=0,
\end{equation}
where 
\[
g_K(u)=g(\rho_K(u)u),
\qquad u\in S^{n-1}.
\]

\begin{lem}\label{lem:gK-integrability}
Let $g\in L^1_{\mathrm{loc}}(\mathbb R^n)$ and
$K\in\mathcal K_o^n$. Then $g_K\in L^1(S^{n-1})$ if and only if
$g\in L^1(\partial K,\mathcal H^{n-1})$.
\end{lem}

\begin{proof}
By \cite[Proposition 2.5]{KL2023},
\[
\int_{\partial K}|g(x)|\,d\mathcal H^{n-1}(x)
=
\int_{S^{n-1}}
|g(\rho_K(u)u)|
\frac{\rho_K(u)^n}
{h_K(n_K(\rho_K(u)u))}
\,du.
\]
Since $K\in\mathcal K_o^n$, there exist constants $c,C>0$ such that
\[
c\leq
\frac{\rho_K(u)^n}
{h_K(n_K(\rho_K(u)u))}
\leq C
\]
for almost every $u\in S^{n-1}$. It follows that
\[
c\int_{S^{n-1}}|g(\rho_K(u)u)|\,du
\leq
\int_{\partial K}|g(x)|\,d\mathcal H^{n-1}(x)
\leq
C\int_{S^{n-1}}|g(\rho_K(u)u)|\,du.
\]
By the definition of $g_K$, this becomes
\[
c\int_{S^{n-1}}|g_K(u)|\,du
\leq
\int_{\partial K}|g(x)|\,d\mathcal H^{n-1}(x)
\leq
C\int_{S^{n-1}}|g_K(u)|\,du.
\]
Thus, the finiteness of one of the two integrals implies the finiteness
of the other, and the proof is complete.
\end{proof}

\begin{lem}\label{trace-implies-integrability}
Let $g\in L^1_{\mathrm{loc}}(\mathbb R^n)$ and let
$K\in\mathcal K_o^n$. If $K\in\mathcal K_o^n$ is
trace-regular with respect to $g$,
then $
g_K\in L^1(S^{n-1}),
$
and hence $K\in\mathcal K_g^n$.
\end{lem}

\begin{proof}
Since $o\in\operatorname{int}K$, there exist constants
$0<c_0<C_0<\infty$ such that
\[
c_0\leq \rho_K(u)\leq C_0
\qquad\text{for all }u\in S^{n-1}.
\]
Fix $0<\varepsilon<c_0/2$. Then
\[
\frac{c_0}{2}
\leq \rho_K(u)-\varepsilon
\leq \rho_K(u)+\varepsilon
\leq C_0+\frac{c_0}{2}.
\]
Since $g\in L^1_{\mathrm{loc}}(\mathbb R^n)$, polar coordinates give
\[
\int_{S^{n-1}}
\int_{\frac{c_0}{2}}^{C_0+\frac{c_0}{2}}
|g(ru)|\,r^{n-1}\,dr\,du<\infty.
\]
Moreover, $r\geq c_0/2$ on this interval, and hence
\begin{equation}\label{shell-integrability}
\begin{aligned}
\int_{S^{n-1}}
\int_{\rho_K(u)-\varepsilon}^{\rho_K(u)+\varepsilon}
|g(ru)|\,dr\,du\leq
\left(\frac{2}{c_0}\right)^{n-1}
\int_{S^{n-1}}
\int_{\frac{c_0}{2}}^{C_0+\frac{c_0}{2}}
|g(ru)|\,r^{n-1}\,dr\,du
<\infty.
\end{aligned}
\end{equation}

Suppose, to the contrary, that
\[
\int_{S^{n-1}}|g(\rho_K(u)u)|\,du=\infty.
\]
By the triangle inequality,
\[
|g(ru)-g(\rho_K(u)u)|
\geq |g(\rho_K(u)u)|-|g(ru)|.
\]
Together with \eqref{shell-integrability}, this yields
\[
\begin{aligned}
&\int_{S^{n-1}}
\int_{\rho_K(u)-\varepsilon}^{\rho_K(u)+\varepsilon}
|g(ru)-g(\rho_K(u)u)|\,dr\,du\\
&\qquad\geq
2\varepsilon
\int_{S^{n-1}}|g(\rho_K(u)u)|\,du\\
&\qquad\quad-
\int_{S^{n-1}}
\int_{\rho_K(u)-\varepsilon}^{\rho_K(u)+\varepsilon}
|g(ru)|\,dr\,du
=\infty.
\end{aligned}
\]
This contradicts \eqref{var}. Therefore,
\[
\int_{S^{n-1}}|g(\rho_K(u)u)|\,du<\infty.
\]
By \eqref{gK}, we conclude that
\[
g_K\in L^1(S^{n-1})
\qquad\text{and}\qquad
K\in\mathcal K_g^n.
\]
\end{proof}

We call $g$ radially non-increasing if
\[
g(tx)\leq g(x)
\qquad\text{for every }x\in\mathbb R^n\setminus\{o\}
\text{ and }t\geq1,
\]
and strictly radially decreasing if
\[
g(tx)<g(x)
\qquad\text{for every }x\in\mathbb R^n\setminus\{o\}
\text{ and }t>1.
\]
We say that $g$ is pointwise $\alpha$-homogeneous if
\[
g(ru)=r^\alpha g(u),
\qquad r>0,\ u\in S^{n-1}.
\]

\begin{lem}\label{trace-examples}
Let $g\in L^1_{\mathrm{loc}}(\mathbb R^n)$ be nonnegative. The following
statements hold.
\begin{enumerate}[label=\textup{(\roman*)}]
\item If $g$ is pointwise $\alpha$-homogeneous, then every
$K\in\mathcal K_o^n$ satisfies \eqref{radial-trace}. In particular,
$K\in\mathcal K_g^n$.

\item If $g$ is radially non-increasing and
$r\mapsto g(ru)$ is continuous at $r=\rho_K(u)$ for almost every
$u\in S^{n-1}$, then $K$ satisfies \eqref{radial-trace}. In
particular, $K\in\mathcal K_g^n$.
\end{enumerate}
\end{lem}

\begin{proof}
Since $o\in\operatorname{int}K$, there exist constants
$0<c_0<C_0<\infty$ such that
\[
c_0\leq \rho_K(u)\leq C_0
\qquad\text{for all }u\in S^{n-1}.
\]

For (i), the pointwise $\alpha$-homogeneity of $g$ and polar
coordinates give
\[
\int_{\{1<|x|<2\}}g(x)\,dx
=
\left(\int_1^2 r^{n+\alpha-1}\,dr\right)
\left(\int_{S^{n-1}}g(u)\,du\right)<\infty.
\]
Since the radial factor is finite and strictly positive, we obtain
\begin{equation}\label{sph}
\int_{S^{n-1}}g(u)\,du<\infty.
\end{equation}

Fix $0<\varepsilon<c_0/2$. Then
\[
r,\rho_K(u)\in
\left[\frac{c_0}{2},C_0+\frac{c_0}{2}\right]
\]
whenever
$r\in[\rho_K(u)-\varepsilon,\rho_K(u)+\varepsilon]$. On this
interval,
\[
\left|\frac{d}{ds}s^\alpha\right|
=|\alpha|s^{\alpha-1}\leq C
\]
for some constant $C>0$ independent of $u$ and $\varepsilon$.
Hence, by the mean value theorem,
\[
|r^\alpha-\rho_K(u)^\alpha|
\leq C|r-\rho_K(u)|.
\]
Using the change of variables
$s=r-\rho_K(u)$, we obtain
\begin{align*}
&\int_{\rho_K(u)-\varepsilon}^{\rho_K(u)+\varepsilon}
|g(ru)-g_K(u)|\,dr\\
&\quad=
g(u)
\int_{\rho_K(u)-\varepsilon}^{\rho_K(u)+\varepsilon}
|r^\alpha-\rho_K(u)^\alpha|\,dr\\
&\quad\leq
Cg(u)
\int_{\rho_K(u)-\varepsilon}^{\rho_K(u)+\varepsilon}
|r-\rho_K(u)|\,dr\\
&\quad=
Cg(u)\int_{-\varepsilon}^{\varepsilon}|s|\,ds
=
Cg(u)\varepsilon^2.
\end{align*}
Consequently,
\[
\begin{aligned}
\frac{1}{\varepsilon}
\int_{S^{n-1}}
\int_{\rho_K(u)-\varepsilon}^{\rho_K(u)+\varepsilon}
|g(ru)-g_K(u)|\,dr\,du\leq
C\varepsilon\int_{S^{n-1}}g(u)\,du.
\end{aligned}
\]
Together with \eqref{sph}, this implies
\[
\lim_{\varepsilon\to0^+}
\frac{1}{\varepsilon}
\int_{S^{n-1}}
\int_{\rho_K(u)-\varepsilon}^{\rho_K(u)+\varepsilon}
|g(ru)-g_K(u)|\,dr\,du
=0.
\]
Thus $K$ satisfies \eqref{radial-trace}. By Lemma
\ref{trace-implies-integrability}, we have
$K\in\mathcal K_g^n$.

For (ii), choose $0<a<c_0/2$. Since
$g\in L^1_{\mathrm{loc}}(\mathbb R^n)$, polar coordinates yield
\[
\int_{\{a<|x|<2a\}}g(x)\,dx
=
\int_a^{2a}r^{n-1}
\left(\int_{S^{n-1}}g(ru)\,du\right)\,dr<\infty.
\]
Hence, for almost every $r\in(a,2a)$,
\[
\int_{S^{n-1}}g(ru)\,du<\infty.
\]
Choose such an $a_0\in(a,2a)$. Then $a_0<c_0$. Since $g$ is
radially non-increasing,
\[
g(ru)\leq g(a_0u)
\qquad\text{for all }u\in S^{n-1},\ r\geq a_0.
\]
Choose
\[
0<\varepsilon<\min\left\{\frac{c_0}{2},c_0-a_0\right\}.
\]
Then
\[
r\geq \rho_K(u)-\varepsilon
\geq c_0-\varepsilon>a_0
\]
for every
$r\in[\rho_K(u)-\varepsilon,\rho_K(u)+\varepsilon]$. Therefore,
\[
g(ru)\leq g(a_0u)
\qquad\text{and}\qquad
g(\rho_K(u)u)\leq g(a_0u),
\]
and hence
\[
|g(ru)-g(\rho_K(u)u)|
\leq 2g(a_0u).
\]
It follows that
\[
\begin{aligned}
\frac{1}{\varepsilon}
\int_{S^{n-1}}
\int_{\rho_K(u)-\varepsilon}^{\rho_K(u)+\varepsilon}
|g(ru)-g(\rho_K(u)u)|\,dr\,du\leq
4\int_{S^{n-1}}g(a_0u)\,du<\infty.
\end{aligned}
\]

Moreover, for almost every $u\in S^{n-1}$, the assumed continuity
of $r\mapsto g(ru)$ at $r=\rho_K(u)$ gives
\[
\lim_{\varepsilon\to0^+}
\frac{1}{\varepsilon}
\int_{\rho_K(u)-\varepsilon}^{\rho_K(u)+\varepsilon}
|g(ru)-g(\rho_K(u)u)|\,dr
=0.
\]
Thus the normalized inner integral converges to zero for almost
every $u$ and is dominated by the integrable function
$4g(a_0u)$. Hence, by the dominated convergence theorem,
\[
\lim_{\varepsilon\to0^+}
\frac{1}{\varepsilon}
\int_{S^{n-1}}
\int_{\rho_K(u)-\varepsilon}^{\rho_K(u)+\varepsilon}
|g(ru)-g(\rho_K(u)u)|\,dr\,du
=0.
\]
Therefore, $K$ satisfies \eqref{radial-trace}. By Lemma
\ref{trace-implies-integrability}, we conclude that
$K\in\mathcal K_g^n$.
\end{proof}

\section{Variational formulas}\label{5}
We use the following radial variation lemma, see \cite[Lemma 2.1]{Lj2022}, \cite[Lemma 6.6]{KL2023}, and \cite[Lemma 3.1]{FLX2025}.
\begin{lem}\label{p1}
Let $p\in\mathbb R$, $K\in\mathcal K_o^n$, and
$f\in C(S^{n-1})$. For $t\in\mathbb R$, define
\begin{equation}\label{htd}
h_t(u)=
\begin{cases}
\big(h_K(u)^p+t f(u)\big)^{1/p},& p\ne0,\\[2mm]
h_K(u)e^{t f(u)},& p=0.
\end{cases}
\end{equation}
Then, there exist a small $\delta>0$ and $M>0$ such that,
for $|t|<\delta$,
\begin{equation*}
h_t(u)>0
\qquad\text{for all }u\in S^{n-1},
\end{equation*}
and
\begin{equation*}
\big|\rho_{[h_t]}(u)-\rho_K(u)\big|
\le M|t|,
\qquad u\in S^{n-1}.
\end{equation*}
Additionally, one has, for almost all $u\in S^{n-1}$,
\begin{equation}\label{pv}
\left.
\frac{d\rho_{[h_t]}(u)}{dt}
\right|_{t=0}
=
\lim_{t\to0}
\frac{\rho_{[h_t]}(u)-\rho_K(u)}{t}
=
\begin{cases}
\dfrac{1}{p}
f\big(\alpha_K(u)\big)
h_K\big(\alpha_K(u)\big)^{-p}
\rho_K(u),& p\ne0,\\[3mm]
f\big(\alpha_K(u)\big)\rho_K(u),& p=0.
\end{cases}
\end{equation}
\end{lem}

\begin{thm}\label{pda0}
Let $p,q\in\mathbb R$, $K\in\mathcal K_o^n$, $Q,C\in\mathcal S_o^n$, and $f\in C(S^{n-1})$. Assume $g\in L^1_{\mathrm{loc}}(\mathbb R^n)$ and that $K$ satisfies \eqref{radial-trace}. For $h_t$ in \eqref{htd},
\begin{align}\label{Vf}
\left.
\frac{d}{dt}\widetilde D_q^\mu([h_t],Q,C)
\right|_{t=0}
&=
\begin{cases}
\displaystyle
\frac1p
\int_{S^{n-1}}f(v)\,
d\widetilde C_{p,q}^\mu(K,Q,v),
& p\ne0,\\[3mm]
\displaystyle
\int_{S^{n-1}}f(v)\,
d\widetilde C_{0,q}^\mu(K,Q,v),
& p=0.
\end{cases}
\end{align}
\end{thm}

\begin{proof}
Write $\rho_t=\rho_{[h_t]}$, $\rho_0=\rho_K$, and
\[
w(u,r)=\rho_Q(u)^{n-q}r^{q-1}.
\]

Lemma~\ref{p1} provides $\delta_0>0$ and $M>0$ such that, for $|t|<\delta_0$,
\[
h_t(u)>0\qquad\text{for all }u\in S^{n-1},
\]
and
\[
|\rho_t(u)-\rho_0(u)|\le M|t|\qquad\text{for all }u\in S^{n-1}.
\]
Since $K\in\mathcal K_o^n$, the radial function $\rho_0$ is continuous and strictly positive on $S^{n-1}$, so
\[
m:=\min_{u\in S^{n-1}}\rho_0(u)>0.
\]
Choose $\delta\in(0,\delta_0)$ so small that $M\delta<m$. Then, for $|t|<\delta$,
\[
\rho_0(u)-M|t|\ge m-M|t|>0\qquad\text{for all }u\in S^{n-1}.
\]
Fix such a $\delta$ and $M$. For $|t|<\delta$, the radial segment joining $\rho_0(u)$ and $\rho_t(u)$ lies in the compact annulus
\[
[a,b]:=[m-M\delta,\ \max_{S^{n-1}}\rho_0+M\delta]\subset(0,\infty),
\]
uniformly in $u\in S^{n-1}$. Hence all integrals below are taken over a compact subset of $\mathbb R^n\setminus\{0\}$.

By \eqref{Vqmu}, for $0<|t|<\delta$,
\begin{align*}
&\frac{\widetilde D_q^\mu([h_t],Q,C)-\widetilde D_q^\mu(K,Q,C)}{t}\\
&\quad=\frac1{nt}\int_{S^{n-1}}\int_{\rho_0(u)}^{\rho_t(u)}w(u,r)g(ru)\,dr\,du\\
&\quad=\frac1n\int_{S^{n-1}}w(u,\rho_0(u))g(\rho_0(u)u)
\frac{\rho_t(u)-\rho_0(u)}{t}\,du+R_t,
\end{align*}
where
\begin{equation}\label{Rt}
R_t
=
\frac1{nt}\int_{S^{n-1}}\int_{\rho_0(u)}^{\rho_t(u)}
\Bigl[w(u,r)g(ru)-w(u,\rho_0(u))g(\rho_0(u)u)\Bigr]\,dr\,du.
\end{equation}

To estimate $R_t$, we use the decomposition
\begin{align}\label{wu}
&w(u,r)g(ru)-w(u,\rho_0(u))g(\rho_0(u)u)\notag\\
&\quad=w(u,r)\bigl(g(ru)-g(\rho_0(u)u)\bigr)
+\bigl(w(u,r)-w(u,\rho_0(u))\bigr)g(\rho_0(u)u).
\end{align}
Since $[\rho_0(u)-M|t|,\rho_0(u)+M|t|]\subset[a,b]$ for $|t|<\delta$, the function $w$ is bounded on $S^{n-1}\times[a,b]$. Let
\[
\|w\|_\infty=\sup_{u\in S^{n-1},\,r\in[a,b]}|w(u,r)|<\infty.
\]
Define the modulus of continuity of $w$ on $S^{n-1}\times[a,b]$ by
\[
\varpi(\eta)=\sup_{\substack{u\in S^{n-1},\ r,s\in[a,b]\\|r-s|\le\eta}}
|w(u,r)-w(u,s)|,
\qquad \varpi(\eta)\longrightarrow0\ \text{as }\eta\to0.
\]
Then, for $0<|t|<\delta$, it follows from \eqref{Rt} and
\eqref{wu} that
\begin{align*}
n|R_t|
\le{}&
\frac{\|w\|_\infty}{|t|}
\int_{S^{n-1}}\int_{\rho_0(u)-M|t|}^{\rho_0(u)+M|t|}
\bigl|g(ru)-g(\rho_0(u)u)\bigr|\,dr\,du\\
&\quad
+\frac{\varpi(M|t|)}{|t|}
\int_{S^{n-1}}\int_{\rho_0(u)-M|t|}^{\rho_0(u)+M|t|}
g(\rho_0(u)u)\,dr\,du.
\end{align*}
The length of the inner interval is at most $2M|t|$, so
\begin{equation}\label{Rt-est}
n|R_t|
\le
\frac{\|w\|_\infty}{|t|}\int_{S^{n-1}}\int_{\rho_0(u)-M|t|}^{\rho_0(u)+M|t|}
\bigl|g(ru)-g(\rho_0(u)u)\bigr|\,dr\,du
+2M\varpi(M|t|)\|g_K\|_{L^1(S^{n-1})}.
\end{equation}

Set $\varepsilon=M|t|$. Then $\varepsilon\to0$ as $t\to0$. By the trace condition \eqref{radial-trace} at $K$, the first term in \eqref{Rt-est} tends to zero. The second term tends to zero because $\varpi(M|t|)\to0$ and $g_K\in L^1(S^{n-1})$ by Lemma \ref{trace-implies-integrability}. Consequently,
\[
R_t\longrightarrow0\qquad\text{as }t\to0.
\]

For the main integral, Lemma~\ref{p1} gives the uniform bound
\[
\left|w(u,\rho_0(u))g(\rho_0(u)u)\frac{\rho_t(u)-\rho_0(u)}{t}\right|
\le
M\|w\|_\infty g(\rho_0(u)u),
\qquad u\in S^{n-1},
\]
and $M\|w\|_\infty g(\rho_0(u)u)\in L^1(S^{n-1})$. Hence, by the dominated convergence theorem and \eqref{pv},
\[
\left.\frac{d}{dt}\widetilde D_q^\mu([h_t],Q,C)\right|_{t=0}
=
\frac1n\int_{S^{n-1}}w(u,\rho_0(u))g(\rho_0(u)u)\rho^{\prime}_0(u)\,du,
\]
where $\rho^{\prime}_0(u)=\left.\dfrac{d\rho_{t}(u)}{dt}\right|_{t=0}$ denotes the almost-everywhere derivative in \eqref{pv}.

Substituting \eqref{pv} and using $g(\rho_0(u)u)=g_K(u)$ yields
\[
\left.
\frac{d}{dt}\widetilde D_q^\mu([h_t],Q,C)
\right|_{t=0}=
\begin{cases}
\dfrac1{np}\displaystyle\int_{S^{n-1}}
f(\alpha_K(u))h_K(\alpha_K(u))^{-p}
\rho_K(u)^q\rho_Q(u)^{n-q}g_K(u)\,du,&p\ne0,\\[3mm]
\dfrac1n\displaystyle\int_{S^{n-1}}
f(\alpha_K(u))\rho_K(u)^q\rho_Q(u)^{n-q}g_K(u)\,du,&p=0.
\end{cases}
\]
The radial representation \eqref{23} then gives \eqref{Vf}.
\end{proof}

Theorem~\ref{main} follows by taking $f=h_L^p$ for $p\ne0$, and $f=\log h_L$ for $p=0$.
\begin{cor}\label{mm}
Let $q\in\mathbb R$, $K,L\in\mathcal K_o^n$, and $Q,C\in\mathcal S_o^n$. Assume $g\in L^1_{\mathrm{loc}}(\mathbb R^n)$ and that $K$ satisfies \eqref{radial-trace}. Then
\begin{equation}\label{vv}
\left.\frac{d}{dt}\widetilde D_q^\mu((1-t)\cdot K+_0t\cdot L,Q,C)\right|_{t=0}
=\int_{S^{n-1}}\log\frac{h_L}{h_K}\,d\widetilde C_{0,q}^\mu(K,Q).
\end{equation}
\end{cor}
\begin{proof}
Apply Theorem~\ref{pda0} with $p=0$ and $f=\log(h_L/h_K)$.
\end{proof}

Recall from \eqref{VD} that
\[
\widetilde D_q^\mu(K,Q,C)
=
\widetilde V_q^\mu(K,Q)
-
\widetilde V_q^\mu(C,Q)
\]
whenever $|x|^{q-n}g(x)\in L^1_{\mathrm{loc}}(\mathbb R^n)$.
Thus, Theorem~\ref{pda0} immediately yields the following
corollary.

\begin{cor}\label{pda1}
Let $p,q\in\mathbb R$, $K\in\mathcal K_o^n$,
$Q\in\mathcal S_o^n$, and $f\in C(S^{n-1})$. Assume that
$|x|^{q-n}g(x)\in L^1_{\mathrm{loc}}(\mathbb R^n)$ and that
$K$ satisfies \eqref{radial-trace}. Then
\begin{equation}\label{Vf1}
\left.\frac{d}{dt}\widetilde V_q^\mu([h_t],Q)\right|_{t=0}
=
\begin{cases}
\dfrac1p\displaystyle\int_{S^{n-1}}
f(v)\,d\widetilde C_{p,q}^\mu(K,Q,v),
& p\ne0,\\[6pt]
\displaystyle\int_{S^{n-1}}
f(v)\,d\widetilde C_{0,q}^\mu(K,Q,v),
& p=0.
\end{cases}
\end{equation}
\end{cor}

Combining \eqref{s1} and \eqref{muK} with Corollary \ref{pda1},
taking $f=h_L^p$ when $p\ne0$ and $f=\log h_L$ when $p=0$,
we obtain the following weighted volume variation formula,
valid for all $p\in\mathbb R$.

\begin{cor}\label{mu22}
Let $p\in\mathbb R$ and $K,L\in\mathcal K_o^n$. Assume that
$g\in L^1_{\mathrm{loc}}(\mathbb R^n)$ and that $K$
satisfies \eqref{radial-trace}. Then
\begin{equation}\label{cp}
\left.\frac{d}{dt}\mu(K+_p t\cdot L)\right|_{t=0}
=
\int_{S^{n-1}}h_L^{\bar p}(v)\,dS_p^\mu(K,v).
\end{equation}
\end{cor}

\section{The even $(p,q)$-weighted dual Minkowski problem}\label{6}

\begin{thm}\label{main2}
Let $p>0$, $q\in\mathbb R$, and $\beta>0$. Let
$Q\in\mathcal S_e^n$, and let $\mu$ be a measure with density $g$
belonging to $\Lambda_{p,q,\beta}$. Assume that every
$K\in\mathcal K_o^n$ satisfies \eqref{radial-trace}. Let $\nu$ be
a finite even Borel measure on $S^{n-1}$ that is not concentrated
on any great subsphere. Then there exists $K_0\in\mathcal K_e^n$
such that
\begin{equation}\label{eq:minkowski-equation2}
d\nu
=
\frac{1}{p}
\widetilde V_q^\mu(K_0,Q)^{\beta-1}
\,d\widetilde C_{p,q}^\mu(K_0,Q,\cdot).
\end{equation}
\end{thm}

\begin{proof}
Define the functional $\Psi_\beta$ on
$C_e^+(S^{n-1})$ by
\begin{equation}\label{Psi}
\Psi_\beta(f)
=
\frac{1}{\beta}
\widetilde V_q^\mu([f],Q)^\beta
-
\int_{S^{n-1}}f^p\,d\nu,
\end{equation}
where $[f]$ denotes the Wulff shape associated with $f$. Set
\begin{equation}\label{ab}
a_Q=\min_{S^{n-1}}\rho_Q^{n-q}>0,
\qquad
b_Q=\max_{S^{n-1}}\rho_Q^{n-q}<\infty.
\end{equation}

We first prove that $\Psi_\beta$ is continuous on
$C_e^+(S^{n-1})$. Let $f_i,f\in C_e^+(S^{n-1})$ with
$f_i\to f$ uniformly on $S^{n-1}$. Then
$[f_i]\to[f]$ in the Hausdorff metric by
\cite[Lemma 7.5.2]{Schneider2014}. Since $f$ is strictly positive,
there exist $0<r_0<R_0<\infty$ such that, for all sufficiently
large $i$,
\[
r_0B\subset [f_i],\ [f]\subset R_0B.
\]
Denote the symmetric difference by
\[
[f_i]\triangle[f]
:=
([f_i]\setminus[f])\cup([f]\setminus[f_i]).
\]
Then
\[
[f_i]\triangle[f]\subset R_0B\setminus r_0B.
\]
By \eqref{Vqmu1},
\begin{align*}
\left|
\widetilde V_q^\mu([f_i],Q)
-
\widetilde V_q^\mu([f],Q)
\right|
&\leq
\frac{1}{n}
\int_{[f_i]\triangle[f]}
\rho_Q^{n-q}(x)g(x)\,dx\\
&\leq
\frac{b_Q}{n}
\int_{[f_i]\triangle[f]}
|x|^{q-n}g(x)\,dx.
\end{align*}
Since $[f_i]\to[f]$ in the Hausdorff metric,
\[
\lambda_n([f_i]\triangle[f])\to0.
\]
Moreover, $|x|^{q-n}g(x)$ is integrable on the compact annulus
$R_0B\setminus r_0B$ by \eqref{admissibility}. Thus, by the
absolute continuity of the Lebesgue integral,
\[
\int_{[f_i]\triangle[f]}
|x|^{q-n}g(x)\,dx\to0.
\]
Therefore,
\[
\widetilde V_q^\mu([f_i],Q)
\to
\widetilde V_q^\mu([f],Q).
\]
Since $[f]$ contains a neighborhood of the origin, the
small-radius condition in \eqref{admissibility} gives
\[
\widetilde V_q^\mu([f],Q)>0.
\]
Consequently,
\[
\widetilde V_q^\mu([f_i],Q)^\beta
\to
\widetilde V_q^\mu([f],Q)^\beta.
\]
Furthermore, $f_i^p\to f^p$ uniformly on $S^{n-1}$, and hence
\[
\int_{S^{n-1}}f_i^p\,d\nu
\to
\int_{S^{n-1}}f^p\,d\nu,
\]
since $\nu$ is finite. Thus
\[
\Psi_\beta(f_i)\to\Psi_\beta(f),
\]
and $\Psi_\beta$ is continuous.

Next, recall that $h_{[f]}\leq f$ and $[h_{[f]}]=[f]$. Since
$p>0$,
\[
\Psi_\beta(h_{[f]})\geq\Psi_\beta(f).
\]
Moreover, since $f$ is even, $[f]$ is origin-symmetric and
$h_{[f]}\in C_e^+(S^{n-1})$. Thus it suffices to maximize
$\Psi_\beta(h_K)$ over $K\in\mathcal K_e^n$.

For $r>0$, \eqref{ab} gives
\[
\frac{a_Q}{n}A_q(r)
\leq
\widetilde V_q^\mu(rB,Q)
\leq
\frac{b_Q}{n}A_q(r),
\]
where $A_q$ is defined in \eqref{admissibility}. By the
small-radius condition in \eqref{admissibility},
\[
\Psi_\beta(h_{rB})
=
\frac{1}{\beta}\widetilde V_q^\mu(rB,Q)^\beta
-r^p\nu(S^{n-1})>0
\]
for all sufficiently small $r>0$. Hence
\[
M:=\sup_{K\in\mathcal K_e^n}\Psi_\beta(h_K)>0.
\]

Let $\{K_i\}\subset\mathcal K_e^n$ be a maximizing sequence such
that
\[
\Psi_\beta(h_{K_i})\to M.
\]
Set
\[
R_i=\max_{u\in S^{n-1}}h_{K_i}(u).
\]
We claim that $\{R_i\}$ is bounded. Otherwise, after passing to a
subsequence, $R_i\to\infty$. Choose $v_i\in S^{n-1}$ such that
$h_{K_i}(v_i)=R_i$. Since $K_i$ is origin-symmetric,
\[
[-R_iv_i,R_iv_i]\subset K_i,
\]
and therefore
\begin{equation}\label{ge}
h_{K_i}(u)\geq R_i|v_i\cdot u|,
\qquad u\in S^{n-1}.
\end{equation}
Since $\nu$ is not concentrated on any great subsphere, by
compactness of $S^{n-1}$ there exists $c_{p,\nu}>0$ such that
\begin{equation}\label{cpnu}
\int_{S^{n-1}}|v\cdot u|^p\,d\nu(u)
\geq c_{p,\nu},
\qquad v\in S^{n-1}.
\end{equation}
Since $K_i\subset R_iB$, the monotonicity of
$\widetilde V_q^\mu(\cdot,Q)$, together with \eqref{ge} and
\eqref{cpnu}, yields
\begin{align}
\Psi_\beta(h_{K_i})
&\leq
\frac{1}{\beta}\widetilde V_q^\mu(R_iB,Q)^\beta
-c_{p,\nu}R_i^p \notag\\
&\leq
\frac{1}{\beta}
\left(\frac{b_Q}{n}A_q(R_i)\right)^\beta
-c_{p,\nu}R_i^p.
\label{boundedness-estimate}
\end{align}
By the large-radius condition in \eqref{admissibility},
\[
\lim_{r\to\infty}
\left[
\frac{1}{\beta}
\left(\frac{b_Q}{n}A_q(r)\right)^\beta
-c_{p,\nu}r^p
\right]
=-\infty.
\]
This contradicts $\Psi_\beta(h_{K_i})\to M>0$. Hence
$\{K_i\}$ is uniformly bounded. In particular, there exists
$R>0$ such that $K_i\subset RB$ for all sufficiently large $i$.
Therefore,
\[
\Psi_\beta(h_{K_i})
\leq
\frac{1}{\beta}\widetilde V_q^\mu(RB,Q)^\beta
\leq
\frac{1}{\beta}
\left(\frac{b_Q}{n}A_q(R)\right)^\beta<\infty.
\]
It follows that $M<\infty$.

By Blaschke's selection theorem, after passing to a subsequence,
$K_i\to K_0$ in the Hausdorff metric for some compact convex set
$K_0\subset\mathbb R^n$. We claim that $K_0$ has nonempty interior. Suppose otherwise.
Then $\lambda_n(K_0)=0$. Since $K_i\to K_0$ and the sequence is
uniformly bounded, there exists $R>0$ such that $K_i\subset RB$
for all $i$. Since
$\rho_Q^{n-q}g\in L^1(RB)$, the dominated convergence theorem
gives
\[
\widetilde V_q^\mu(K_i,Q)\to0.
\]
Hence
\[
\limsup_{i\to\infty}\Psi_\beta(h_{K_i})
\leq
\lim_{i\to\infty}
\frac{1}{\beta}
\widetilde V_q^\mu(K_i,Q)^\beta
=0,
\]
contradicting
\[
\Psi_\beta(h_{K_i})\to M>0.
\]
Thus $K_0$ has nonempty interior. Since $K_0$ is origin-symmetric,
we conclude that $K_0\in\mathcal K_e^n$. Since $K_i\to K_0$ in the Hausdorff metric, we have
$h_{K_i}\to h_{K_0}$ uniformly. By the continuity of
$\Psi_\beta$,
\[
\Psi_\beta(h_{K_i})
\to
\Psi_\beta(h_{K_0})=M.
\]
Thus $h_{K_0}$ is a maximizer of $\Psi_\beta$.

It remains to verify the variational equation. Let
$\psi\in C_e(S^{n-1})$ be arbitrary. For all sufficiently small
$|t|$, define
\[
h_t(u)=\bigl(h_{K_0}^p(u)+t\psi(u)\bigr)^{1/p},
\qquad u\in S^{n-1}.
\]
Since $h_{K_0}>0$, we have
$h_t\in C_e^+(S^{n-1})$ for all sufficiently small $|t|$, and
$h_0=h_{K_0}$. By the maximality of $h_{K_0}$,
\[
\left.\frac{d}{dt}\Psi_\beta(h_t)\right|_{t=0}=0.
\]
Using \eqref{Psi} and the variational formula \eqref{Vf1}, we obtain
\begin{align*}
0
&=
\widetilde V_q^\mu(K_0,Q)^{\beta-1}
\left.\frac{d}{dt}
\widetilde V_q^\mu([h_t],Q)\right|_{t=0}
-
\left.\frac{d}{dt}
\int_{S^{n-1}}h_t^p\,d\nu\right|_{t=0}\\
&=
\frac{1}{p}
\widetilde V_q^\mu(K_0,Q)^{\beta-1}
\int_{S^{n-1}}
\psi(v)\,d\widetilde C_{p,q}^\mu(K_0,Q,v)
-
\int_{S^{n-1}}\psi(v)\,d\nu(v).
\end{align*}
Hence
\[
\int_{S^{n-1}}\psi(v)\,d\nu(v)
=
\frac{1}{p}
\widetilde V_q^\mu(K_0,Q)^{\beta-1}
\int_{S^{n-1}}
\psi(v)\,d\widetilde C_{p,q}^\mu(K_0,Q,v)
\]
for every $\psi\in C_e(S^{n-1})$.

Since both $\nu$ and
$\widetilde C_{p,q}^\mu(K_0,Q,\cdot)$ are even, the preceding
identity for all even continuous functions implies
\[
d\nu(\cdot)
=
\frac{1}{p}
\widetilde V_q^\mu(K_0,Q)^{\beta-1}
\,d\widetilde C_{p,q}^\mu(K_0,Q,\cdot).
\]
This proves \eqref{eq:minkowski-equation2}.
\end{proof}

\begin{thm}\label{cor:examples-Lambda}
Let $p > 0$, $q \in \mathbb{R}$, $\beta > 0$, and let $Q\in\mathcal S_e^n$.
Let $\mu$ be a Borel measure on $\mathbb{R}^n$ with an even nonnegative density $g\in L^1_{\mathrm{loc}}(\mathbb R^n)$ such that $\lambda_n(\{x:g(x)>0\})>0$. 
Let $\nu$ be a finite even Borel measure on $S^{n-1}$ that is not concentrated on any great subsphere. 
Suppose that one of the following conditions holds:
\begin{enumerate}[label=\textup{(\roman*)}]
    \item $g$ is pointwise $\alpha$-homogeneous for some
    $\alpha>-q$, and $\beta<\frac{p}{q+\alpha}$;
    \item $q>0$, $g$ is radially non-increasing and continuous on $\mathbb R^n$, and $\beta<\frac{p}{q}$.
\end{enumerate}
Then $\mu\in\Lambda_{p,q,\beta}$, and there exists a convex body
$K_0\in\mathcal K_e^n$ satisfying
\begin{equation*}
\mathrm d\nu
=
\frac{1}{p}\widetilde V_q^\mu(K_0,Q)^{\beta-1}
\,\mathrm d\widetilde C_{p,q}^\mu(K_0,Q,\cdot).
\end{equation*}
\end{thm}

\begin{proof}
By Lemma~\ref{trace-examples}, in both cases the trace-regularity
condition \eqref{radial-trace} holds, and every
$K\in\mathcal K_o^n$ belongs to $\mathcal K_g^n$.
Thus, in view of Theorem~\ref{main2}, it remains to verify that
$\mu\in\Lambda_{p,q,\beta}$.

Recall from \eqref{admissibility} that 
\[
A_q(r)=\int_{rB}|x|^{q-n}g(x)\,\mathrm{d}x.
\]

\medskip
\noindent\textup{(i)}
Suppose that $g$ is pointwise $\alpha$-homogeneous for some
$\alpha>-q$. By polar coordinates,
\[
\begin{aligned}
A_q(r)
&=\int_0^r t^{q-1}
  \int_{S^{n-1}}g(tu)\,\mathrm{d}u\,\mathrm{d}t\\
&=\left(\int_{S^{n-1}}g(u)\,\mathrm{d}u\right)
  \int_0^r t^{q+\alpha-1}\,\mathrm{d}t\\
&=\frac{r^{q+\alpha}}{q+\alpha}
  \int_{S^{n-1}}g(u)\,\mathrm{d}u.
\end{aligned}
\]
Since $q+\alpha>0$ and $g\ge0$, the weighted density
$|x|^{q-n}g(x)$ is locally integrable. Moreover, the
nontriviality assumption and homogeneity imply
\[
\int_{S^{n-1}}g(u)\,\mathrm{d}u>0.
\]
Indeed, if this integral were zero, then $g=0$ almost everywhere
on $S^{n-1}$, and homogeneity would imply that $g=0$ almost
everywhere in $\mathbb R^n$, contradicting
$\lambda_n(\{g>0\})>0$.

Consequently,
\[
A_q(r)=C_0r^{q+\alpha},
\qquad
C_0:=\frac{1}{q+\alpha}
\int_{S^{n-1}}g(u)\,\mathrm{d}u>0.
\]
It follows that
\[
\frac{A_q(r)^\beta}{r^p}
=C_0^\beta r^{\beta(q+\alpha)-p}.
\]
Since $\beta(q+\alpha)<p$, we obtain
\[
\lim_{r\to\infty}\frac{A_q(r)^\beta}{r^p}=0,
\qquad
\lim_{r\to0^+}\frac{A_q(r)^\beta}{r^p}=\infty.
\]
Hence $\mu\in\Lambda_{p,q,\beta}$ in this case.

\medskip
\noindent\textup{(ii)}
Suppose that $q>0$ and that $g$ is radially non-increasing and
continuous on $\mathbb R^n$. The local integrability of
$|x|^{q-n}g(x)$ follows from the continuity of $g$ and the
condition $q>0$.

We first show that $A_q(1)>0$. If $A_q(1)=0$, then the
nonnegativity of $g$ implies that $g=0$ almost everywhere on
$B$. By continuity, $g=0$ throughout $B$. Radial monotonicity
then gives $g=0$ on all of $\mathbb R^n$, contradicting
$\lambda_n(\{g>0\})>0$. Thus $A_q(1)>0$.

For $r>0$, the change of variables $x=ry$ gives
\[
A_q(r)=r^q\int_B |y|^{q-n}g(ry)\,\mathrm{d}y.
\]
By radial monotonicity,
\[
A_q(r)\ge r^q A_q(1),\qquad 0<r\le1,
\]
and
\[
A_q(r)\le r^q A_q(1),\qquad r\ge1.
\]
Therefore, since $\beta q<p$,
\[
\frac{A_q(r)^\beta}{r^p}
\le A_q(1)^\beta r^{q\beta-p}\longrightarrow0
\qquad\text{as }r\to\infty,
\]
whereas
\[
\frac{A_q(r)^\beta}{r^p}
\ge A_q(1)^\beta r^{q\beta-p}\longrightarrow\infty
\qquad\text{as }r\to0^+.
\]
Hence $\mu\in\Lambda_{p,q,\beta}$ in this case as well.

In either case, all the hypotheses of Theorem~\ref{main2} are
satisfied. Applying Theorem \eqref{main2} yields
$K_0\in\mathcal K_e^n$ such that
\[
\mathrm d\nu
=
\frac{1}{p}\widetilde V_q^\mu(K_0,Q)^{\beta-1}
\,\mathrm d\widetilde C_{p,q}^\mu(K_0,Q,\cdot).
\]
This completes the proof.
\end{proof}

\begin{thm}\label{52}
Let $p > 0$, $q \in \mathbb{R}$, and $\beta \ge 1$, and let $Q \in \mathcal{S}_o^n$. 
Let $\mu$ be a Borel measure on $\mathbb{R}^n$ with a positive density $g \in L^1_{\mathrm{loc}}(\mathbb{R}^n)$ satisfying $
|x|^{q-n}g(x) \in L^1_{\mathrm{loc}}(\mathbb{R}^n)$.
For polytopes $P, P' \in \mathcal{K}_o^n$, suppose that
\begin{equation}\label{uni}
\widetilde V_q^\mu(P,Q)^{\beta-1} \widetilde C_{p,q}^\mu(P,Q,\cdot) 
= \widetilde V_q^\mu(P',Q)^{\beta-1} \widetilde C_{p,q}^\mu(P',Q,\cdot).
\end{equation}
Then the following statements hold:
\begin{enumerate}[label=\textup{(\roman*)}]
    \item If $g$ is pointwise $\alpha$-homogeneous for some $\alpha > -q$, then $P = P'$ provided that $(q+\alpha)\beta < p$, while $P'$ is a dilate of $P$ if $(q+\alpha)\beta = p$.

    \item If $g$ is radially non-increasing and continuous on $\mathbb R^n$, then $P = P'$ if $q\beta < p$. If, in addition, $g$ is strictly radially decreasing, then $P = P'$ even when $q\beta = p$.
\end{enumerate}
\end{thm}

\begin{proof}
In either case, Lemma~\ref{trace-examples} gives $g_K\in L^1(S^{n-1})$ for every $K\in\mathcal K_o^n$.
Let $v_1,\ldots,v_m$ denote the outer unit normals of the facets of $P$,
and let $\Delta_i$ be the cone consisting of all rays emanating from the
origin and passing through the facet of $P$ with outer unit normal $v_i$.
By the relation between the radial Gauss image \eqref{ala} and the reverse
radial Gauss image \eqref{sta}, we have
\begin{equation}\label{dis}
\boldsymbol{\alpha}_P^*(v_i)=S^{n-1}\cap\Delta_i,
\qquad
\alpha_P(u)=v_i
\quad\text{for a.e. }u\in S^{n-1}\cap\Delta_i.
\end{equation}
Hence, if $\eta\subset S^{n-1}$ is a Borel set satisfying
$\eta\cap\{v_1,\ldots,v_m\}=\varnothing$, then
$\boldsymbol{\alpha}_P^*(\eta)$ has zero spherical Lebesgue measure.
Consequently, $\widetilde C^\mu_{p,q}(P,Q,\cdot)$ is discrete and
supported on $\{v_1,\ldots,v_m\}$. Combining \eqref{dis} with the
integral representation \eqref{all}, we obtain
\begin{equation*}
\widetilde C^\mu_{p,q}(P,Q,\cdot)=\sum_{i=1}^m c_i(P)\,\delta_{v_i},
\end{equation*}
where $\delta_{v_i}$ denotes the delta measure concentrated at $v_i$, and
\begin{equation}\label{cP}
c_i(P)=\frac{1}{n}h_P(v_i)^{-p}
\int_{S^{n-1}\cap\Delta_i}
\rho_P^q(u)\rho_Q^{\,n-q}(u)\,g(r_P(u))\,du.
\end{equation}
Since $P\in\mathcal K_o^n$ and $Q\in\mathcal S_o^n$, we have $\rho_P>0$
and $\rho_Q>0$ on $S^{n-1}$. Together with the assumption $g>0$, this
yields $c_i(P)>0$ for all $i$. The same holds for $P'$. Since all
coefficients are positive, \eqref{uni} implies that the two measures have
the same support; hence $P$ and $P'$ have exactly the same facet normals
$v_1,\ldots,v_m$. As $o\in\operatorname{int}P\cap\operatorname{int}P'$,
\begin{equation}\label{halfspaces}
P=\bigcap_{i=1}^m\{x\in\mathbb R^n:x\cdot v_i\le h_P(v_i)\},
\qquad
P'=\bigcap_{i=1}^m\{x\in\mathbb R^n:x\cdot v_i\le h_{P'}(v_i)\}.
\end{equation}

Suppose $P\neq P'$. Since \eqref{uni} is symmetric in $P$ and $P'$,
interchanging them if necessary, we may assume that $P\not\subseteq P'$.
By \eqref{halfspaces}, the inclusion $tP\subseteq P'$ is equivalent to
\[
t\,h_P(v_i)\le h_{P'}(v_i),
\qquad i=1,\ldots,m.
\]
Therefore,
\begin{equation}\label{lambda}
\lambda:=\max\{t>0:tP\subseteq P'\}
=\min_{1\le i\le m}\frac{h_{P'}(v_i)}{h_P(v_i)}.
\end{equation}
Then $\lambda>0$. Moreover, $\lambda\ge1$ would imply
$P\subseteq\lambda P\subseteq P'$, contradicting $P\not\subseteq P'$.
Thus $0<\lambda<1$. Choose $i_1$ attaining the minimum in \eqref{lambda},
so that
\begin{equation}\label{contact}
h_{\lambda P}(v_{i_1})=h_{P'}(v_{i_1}),
\qquad\text{and}\qquad
\lambda P\subseteq P'.
\end{equation}
For $u\in S^{n-1}\cap\Delta_{i_1}$, the radial point $\rho_{\lambda P}(u)u$
lies on the facet of $\lambda P$ with normal $v_{i_1}$, so
$\rho_{\lambda P}(u)\,u\cdot v_{i_1}=h_{\lambda P}(v_{i_1})$ and $u\cdot v_{i_1}>0$. Together with \eqref{contact},
\begin{equation}\label{lower}
\frac{h_{P'}(v_{i_1})}{u\cdot v_{i_1}}
=\frac{h_{\lambda P}(v_{i_1})}{u\cdot v_{i_1}}
=\rho_{\lambda P}(u)
\le\rho_{P'}(u).
\end{equation}
On the other hand, since $\rho_{P'}(u)u\in P'$, then
\begin{equation}\label{upper}
\rho_{P'}(u)\le\frac{h_{P'}(v_{i_1})}{u\cdot v_{i_1}}.
\end{equation}
Combining \eqref{lower} and \eqref{upper} gives
\begin{equation}\label{rad}
\rho_{P'}(u)=\frac{h_{P'}(v_{i_1})}{u\cdot v_{i_1}}=\rho_{\lambda P}(u),
\qquad
u\in S^{n-1}\cap\Delta_{i_1}.
\end{equation}
Thus, $\rho_{P'}(u)u$ lies on the facet of $P'$ with normal
$v_{i_1}$. Writing $\Delta_{i_1}'$ for the cone over that facet of $P'$,
we obtain
\begin{equation}\label{coneinclusion}
\Delta_{i_1}\subseteq\Delta_{i_1}'.
\end{equation}
By \eqref{rad}, the integrands in \eqref{cP} for $\lambda P$ and $P'$ coincide on $S^{n-1} \cap \Delta_{i_1}$, and by \eqref{contact} the coefficients $h_{\lambda P}(v_{i_1})^{-p} = h_{P'}(v_{i_1})^{-p}$ coincide as well.
Since the integrand is strictly positive, the cone inclusion
\eqref{coneinclusion} yields
\begin{equation}\label{compare}
c_{i_1}(\lambda P)\le c_{i_1}(P'),
\end{equation}
with equality only if $S^{n-1}\cap\Delta_{i_1}$ and $S^{n-1}\cap\Delta_{i_1}'$ agree up to a set of spherical measure zero. Finally, comparing the coefficients of
$\delta_{v_{i_1}}$ in \eqref{uni} gives
\begin{equation}\label{coeff}
\widetilde V_q^\mu(P,Q)^{\beta-1}c_{i_1}(P)
=
\widetilde V_q^\mu(P',Q)^{\beta-1}c_{i_1}(P').
\end{equation}

\medskip
\noindent\textbf{Case (i).} $g$ is $\alpha$-homogeneous.

Since $g(\lambda x)=\lambda^\alpha g(x)$, $\rho_{\lambda P}=\lambda\rho_P$
and $h_{\lambda P}=\lambda h_P$, formula \eqref{cP} gives
\begin{equation}\label{c-scale}
c_i(\lambda P)=\lambda^{q-p+\alpha}c_i(P),
\qquad i=1,\ldots,m,
\end{equation}
and it follows from \eqref{Vqmu1} that
\begin{equation}\label{V-scale}
\widetilde V_q^\mu(\lambda P,Q)
=\lambda^{q+\alpha}\widetilde V_q^\mu(P,Q).
\end{equation}
Since $\lambda P\subseteq P'$ implies
\begin{equation*}\widetilde V_q^\mu(\lambda P,Q)\le\widetilde V_q^\mu(P',Q)\end{equation*} and
$\beta\ge1$, combining \eqref{c-scale}, \eqref{V-scale}, \eqref{compare}
and \eqref{coeff}, we obtain
\begin{align}
\lambda^{(q+\alpha)\beta-p}\,
\widetilde V_q^\mu(P,Q)^{\beta-1}c_{i_1}(P)
&=\widetilde V_q^\mu(\lambda P,Q)^{\beta-1}c_{i_1}(\lambda P)\notag\\
&\le\widetilde V_q^\mu(P',Q)^{\beta-1}c_{i_1}(\lambda P)\notag\\
&\le\widetilde V_q^\mu(P',Q)^{\beta-1}c_{i_1}(P')\notag\\
&=\widetilde V_q^\mu(P,Q)^{\beta-1}c_{i_1}(P).
\label{equ}
\end{align}Since
$\widetilde V_q^\mu(P,Q)^{\beta-1}c_{i_1}(P)>0$, we conclude that
\begin{equation}\label{key}
\lambda^{(q+\alpha)\beta-p}\le1.
\end{equation}
If $(q+\alpha)\beta<p$, then $(q+\alpha)\beta-p<0$, and $0<\lambda<1$
gives $\lambda^{(q+\alpha)\beta-p}>1$, contradicting \eqref{key}. Hence
$P=P'$.

If $(q+\alpha)\beta=p$, then
$$
\lambda^{(q+\alpha)\beta-p}=1.
$$
Hence the two ends of \eqref{equ} are equal, and therefore every
inequality in \eqref{equ} is an equality. In particular,
\begin{equation*}
c_{i_1}(\lambda P)=c_{i_1}(P').
\end{equation*}
Moreover, if $\beta>1$, then
\begin{equation}\label{V-eq}
\widetilde V_q^\mu(\lambda P,Q)=
\widetilde V_q^\mu(P',Q).
\end{equation}
Since $\lambda P\subseteq P'$ and $g>0$, equality in \eqref{V-eq}
implies
$$
\lambda P=P'.
$$
Indeed, if the inclusion were strict, then
$P'\setminus\lambda P$ would have positive Lebesgue measure, and hence
$$
\widetilde V_q^\mu(P',Q)
>
\widetilde V_q^\mu(\lambda P,Q),
$$
a contradiction. Thus, when $\beta>1$, the desired conclusion already
follows. We may therefore assume that $\beta=1$.

It remains to consider the case $\beta=1$. Equality in
\eqref{compare} implies that
$S^{n-1}\cap\Delta_{i_1}$ and
$S^{n-1}\cap\Delta_{i_1}'$ agree up to a set of spherical measure zero.
Since
$$
S^{n-1}\cap\Delta_{i_1}
\subseteq
S^{n-1}\cap\Delta_{i_1}',
$$
a proper inclusion would yield a set of positive spherical measure, contradicting the equality of the corresponding integrals. Hence
\begin{equation*}
\Delta_{i_1}=\Delta_{i_1}'.
\end{equation*}
Together with \eqref{contact}, this shows that the facets of $\lambda P$
and $P'$ with outer normal $v_{i_1}$ coincide.

We now show that this coincidence propagates to all facets. Let $E$ be
a ridge of the common $v_{i_1}$-facet. Suppose that $E$ is contained in
the facets with outer normals $v_{i_1}$ and $v_{i_2}$ of $\lambda P$.
Since the $v_{i_1}$-facets of $\lambda P$ and $P'$ coincide, $E$ is also
a ridge of $P'$. Let $v_j$ be the outer normal of the other facet of
$P'$ containing $E$.
Since $\lambda P\subseteq P'$,
$$
h_{\lambda P}(v_j)\le h_{P'}(v_j).
$$
On the other hand, $E\subset\lambda P$ and
$E\subset\{x\in\mathbb R^n:x\cdot v_j=h_{P'}(v_j)\}$, so
$$
h_{\lambda P}(v_j)
\ge \max_{x\in E}x\cdot v_j
=h_{P'}(v_j).
$$
Hence
\begin{equation*}
h_{\lambda P}(v_j)=h_{P'}(v_j).
\end{equation*}
Thus the same contact relation as in \eqref{contact} holds with
$v_{i_1}$ replaced by $v_j$.

Applying \eqref{lower}--\eqref{coneinclusion} with $j$ in place of
$i_1$, we obtain
$$
\Delta_j\subseteq\Delta_j'
\qquad\text{and}\qquad
c_j(\lambda P)\le c_j(P').
$$
Since $(q+\alpha)\beta=p$, the two ends of the corresponding version of
\eqref{equ} are equal. Hence all the inequalities are equalities, and
therefore
$$
\Delta_j=\Delta_j'.
$$
Thus the $v_j$-facets of $\lambda P$ and $P'$ also coincide.

Repeating this argument along adjacent facets and using the connectedness
of the facet adjacency graph, we conclude that the corresponding facets
of $\lambda P$ and $P'$ coincide. Hence
$$
h_{\lambda P}(v_i)=h_{P'}(v_i),
\qquad i=1,\ldots,m.
$$
By \eqref{halfspaces},
$$
P'=\lambda P.
$$
Thus $P'$ is a dilate of $P$.

\medskip
\noindent\textbf{Case (ii).} $g$ is radially non-increasing.

We first claim that, for every $0<\lambda<1$,
\begin{equation}\label{eq:V-radial-dec}
\widetilde V_q^\mu(\lambda P,Q)
\ge \lambda^q\widetilde V_q^\mu(P,Q),
\end{equation}
with strict inequality if $g$ is strictly radially decreasing. Indeed, by the change of variables $x=\lambda y$ and the $(-1)$-homogeneity of the radial function of $Q$,
$$
\begin{aligned}
\widetilde V_q^\mu(\lambda P,Q)
&=\frac{1}{n}\int_{\lambda P}\rho_Q^{\,n-q}(x)g(x)\,dx\\
&=\frac{\lambda^q}{n}
  \int_P\rho_Q^{\,n-q}(y)g(\lambda y)\,dy\\
&\ge \frac{\lambda^q}{n}
  \int_P\rho_Q^{\,n-q}(y)g(y)\,dy\\
&=\lambda^q\widetilde V_q^\mu(P,Q),
\end{aligned}
$$
where the inequality follows from
$$
g(\lambda y)\ge g(y),\qquad y\in P,
$$
since $g$ is radially non-increasing and $0<\lambda<1$.
If $g$ is strictly radially decreasing, then
$$
g(\lambda y)>g(y)
\qquad\text{for every }y\in P\setminus\{o\}.
$$
Since $P\setminus\{o\}$ has positive Lebesgue measure, and
$$
\rho_Q^{\,n-q}(y)>0
\qquad\text{for every }y\in P\setminus\{o\}.
$$
Therefore,
$$
\int_P\rho_Q^{\,n-q}(y)
\bigl(g(\lambda y)-g(y)\bigr)\,dy>0,
$$
and hence
\begin{equation}\label{str}
\widetilde V_q^\mu(\lambda P,Q)>
\lambda^q\widetilde V_q^\mu(P,Q).
\end{equation}

Similarly, since
$r_{\lambda P}(u)=\lambda r_P(u)$ and
$h_{\lambda P}(v_{i_1})=\lambda h_P(v_{i_1})$, formula
\eqref{cP} gives
\begin{equation}\label{c-dec}
c_{i_1}(\lambda P)
\ge \lambda^{q-p}c_{i_1}(P),
\end{equation}
where the inequality follows from
$$
g(\lambda r_P(u))\ge g(r_P(u)),
\qquad
u\in S^{n-1}\cap\Delta_{i_1}.
$$
If $g$ is strictly radially decreasing, then
$$
g(\lambda r_P(u))>g(r_P(u))
\qquad\text{for every }u\in S^{n-1}\cap\Delta_{i_1}.
$$
Since $S^{n-1}\cap\Delta_{i_1}$ has positive spherical measure and
$$
\rho_P^q(u)\rho_Q^{\,n-q}(u)>0
\qquad\text{on }S^{n-1}\cap\Delta_{i_1},
$$
the inequality in \eqref{c-dec} is strict.

Combining \eqref{eq:V-radial-dec} (raised to the power $\beta-1\ge0$),
\eqref{c-dec}, \eqref{compare} and \eqref{coeff}, we obtain
\begin{align}
\lambda^{q\beta-p}\,
\widetilde V_q^\mu(P,Q)^{\beta-1}c_{i_1}(P)
&\le\widetilde V_q^\mu(\lambda P,Q)^{\beta-1}c_{i_1}(\lambda P)\notag\\
&\le\widetilde V_q^\mu(P',Q)^{\beta-1}c_{i_1}(\lambda P)\notag\\
&\le\widetilde V_q^\mu(P',Q)^{\beta-1}c_{i_1}(P')\notag\\
&=\widetilde V_q^\mu(P,Q)^{\beta-1}c_{i_1}(P).
\label{chain2}
\end{align} Since
$\widetilde V_q^\mu(P,Q)^{\beta-1}c_{i_1}(P)>0$, we conclude that
\begin{equation}\label{key2}
\lambda^{q\beta-p}\le1.
\end{equation}
If $q\beta<p$, then $\lambda^{q\beta-p}>1$ for $0<\lambda<1$, contradicting \eqref{key2}. Thus $P=P'$.

If $q\beta=p$ and $g$ is strictly radially decreasing, then the first
inequality in \eqref{chain2} is strict. Indeed, when $\beta>1$, this
follows from inequality \eqref{str}; when
$\beta=1$, it follows from the strict inequality in \eqref{c-dec}.
Consequently, \eqref{chain2} yields
$$
\lambda^{q\beta-p}<1.
$$
But $q\beta=p$, so that
$$
\lambda^{q\beta-p}=\lambda^0=1,
$$
a contradiction. Hence $P=P'$.

\end{proof}

For completeness, we note that under the positivity assumptions of 
Theorem~\ref{52}, any solution corresponding to a finitely supported measure 
is a polytope. Indeed, the strict positivity of $g$, together with 
\eqref{CWqQ} and \eqref{Smu}, implies that $\widetilde{C}_{p,q}^\mu(K,Q,\cdot)$ 
and $S_1^\mu(K,\cdot)$ share the same null sets. Consequently, the latter is 
finitely supported, which forces $K$ to be a polytope. Theorem~\ref{52} then 
ensures uniqueness whenever a solution exists. Combining this with 
Theorems~\ref{main2} and~\ref{cor:examples-Lambda} completes the proof of 
Theorem~\ref{thm:exist-uniqueness}.

\section{Affine properties and weighted dual mixed volumes}\label{7}

Suppose $p\in\mathbb R$ and $\nu$ is a Borel measure on
$S^{n-1}$. For $\phi\in\mathrm{GL}(n)$, the $L_p$ image of $\nu$
under $\phi$, denoted by $\phi_{p^\dashv}\nu$, is the Borel measure
defined by
\begin{equation}\label{B}
\int_{S^{n-1}} f(v)\,d\phi_{p^\dashv}\nu(v)
=
\int_{S^{n-1}}
|\phi^{-1}v|^p
f\bigl(\overline{\phi^{-1}v}\bigr)\,d\nu(v),
\end{equation}
for every bounded Borel function $f:S^{n-1}\to\mathbb R$, where $\overline{\phi^{-1}v}=\phi^{-1}v/|\phi^{-1}v|$.
The affine property of the $(p,q)$-weighted dual curvature measures
is established in the following theorem.

\begin{thm}\label{pp1}
Let $p,q\in\mathbb R$, $K\in\mathcal K_o^n$, and
$Q\in\mathcal S_o^n$. Assume that
$g\in L^1_{\mathrm{loc}}(\mathbb R^n)$ and that $\phi K$ satisfies
the trace-regularity condition \eqref{radial-trace} for every
$\phi\in\mathrm{SL}(n)$. Then, for every $\phi\in\mathrm{SL}(n)$,
\begin{equation}\label{SB}
\widetilde{C}^{\mu}_{p,q}(\phi K,\phi Q,\cdot)
=
\phi_{p^\dashv}^t
\widetilde{C}^{\mu^\phi}_{p,q}(K,Q,\cdot),
\end{equation}
where $\mu^\phi$ is the measure with density $g\circ\phi$.
\end{thm}

\begin{proof}
We first verify that $K$ is trace-regular for $g\circ\phi$.
Since $\phi\in\mathrm{SL}(n)$, there exist constants
$0<c_\phi\le C_\phi<\infty$ such that
\[
c_\phi\le|\phi u|\le C_\phi,
\qquad u\in S^{n-1}.
\]
For $v=\phi u/|\phi u|$, equations \eqref{hom} and \eqref{phin}
give
\[
\rho_{\phi K}(v)=|\phi u|\rho_K(u),
\qquad
dv=|\phi u|^{-n}\,du.
\]
Define
\[
E_\varepsilon(g,K)
:=
\int_{S^{n-1}}
\int_{\rho_K(u)-\varepsilon}^{\rho_K(u)+\varepsilon}
|g(ru)-g(\rho_K(u)u)|\,dr\,du.
\]
Using the changes of variables
$v=\phi u/|\phi u|$ and $s=|\phi u|r$, we obtain
\begin{align*}
E_\varepsilon(g\circ\phi,K)
&=
\int_{S^{n-1}}
\int_{\rho_K(u)-\varepsilon}^{\rho_K(u)+\varepsilon}
|g(r\phi u)-g(\rho_K(u)\phi u)|\,dr\,du\\
&=
\int_{S^{n-1}}|\phi u|^{n-1}
\int_{|\phi u|(\rho_K(u)-\varepsilon)}
^{|\phi u|(\rho_K(u)+\varepsilon)}
|g(sv)-g(\rho_{\phi K}(v)v)|\,ds\,dv\\
&\le
C_\phi^{n-1}
\int_{S^{n-1}}
\int_{\rho_{\phi K}(v)-C_\phi\varepsilon}
^{\rho_{\phi K}(v)+C_\phi\varepsilon}
|g(sv)-g(\rho_{\phi K}(v)v)|\,ds\,dv\\
&=
C_\phi^{n-1}E_{C_\phi\varepsilon}(g,\phi K).
\end{align*}
Consequently,
\[
\frac{E_\varepsilon(g\circ\phi,K)}{\varepsilon}
\le
C_\phi^n
\frac{E_{C_\phi\varepsilon}(g,\phi K)}
{C_\phi\varepsilon}
\longrightarrow0
\qquad\text{as }\varepsilon\to0^+,
\]
because $\phi K$ is trace-regular for $g$. Thus $K$ is
trace-regular for $g\circ\phi$. Lemma~\ref{trace-implies-integrability}
then yields
\[
g\circ\phi\in L^1(\partial K,\mathcal H^{n-1}).
\]
Moreover, Lemmas~\ref{lem:gK-integrability} and
\ref{trace-implies-integrability}, together with the trace-regularity
assumption, imply that
\[
g\in L^1(\partial(\phi K),\mathcal H^{n-1})
\qquad\text{for every }\phi\in\mathrm{SL}(n).
\]
Hence the boundary integrability required for the variational
formulas holds for both $g$ and $g\circ\phi$.

By \eqref{Vqmu}, the change of variables $x=\phi y$, and
\eqref{phin}, we have
\begin{align}\label{aff}
\widetilde{D}^{\mu}_{q}(\phi K,\phi Q,\phi C)
&=
\frac1n\int_{\phi K\setminus\phi C}
\rho_{\phi Q}(x)^{n-q}g(x)\,dx
-
\frac1n\int_{\phi C\setminus\phi K}
\rho_{\phi Q}(x)^{n-q}g(x)\,dx
\notag\\
&=
\frac1n\int_{\phi K\setminus\phi C}
\rho_Q(\phi^{-1}x)^{n-q}g(x)\,dx
-
\frac1n\int_{\phi C\setminus\phi K}
\rho_Q(\phi^{-1}x)^{n-q}g(x)\,dx
\notag\\
&=
\frac1n\int_{K\setminus C}
\rho_Q(y)^{n-q}g(\phi y)\,dy
-
\frac1n\int_{C\setminus K}
\rho_Q(y)^{n-q}g(\phi y)\,dy
\notag\\
&=
\widetilde{D}^{\mu^\phi}_{q}(K,Q,C).
\end{align}

For $p\ne0$, it follows from \eqref{aph} and \eqref{aff} that
\begin{align*}
&\lim_{t\to0}
\frac{
\widetilde{D}^{\mu}_{q}
(\phi K+_pt\cdot\phi L,\phi Q,\phi C)
-
\widetilde{D}^{\mu}_{q}(\phi K,\phi Q,\phi C)
}{t}\\
&\quad=
\lim_{t\to0}
\frac{
\widetilde{D}^{\mu}_{q}
\bigl(\phi(K+_pt\cdot L),\phi Q,\phi C\bigr)
-
\widetilde{D}^{\mu}_{q}(\phi K,\phi Q,\phi C)
}{t}\\
&\quad=
\lim_{t\to0}
\frac{
\widetilde{D}^{\mu^\phi}_{q}
(K+_pt\cdot L,Q,C)
-
\widetilde{D}^{\mu^\phi}_{q}(K,Q,C)
}{t}.
\end{align*}
By \eqref{Vv}, we obtain
\[
\int_{S^{n-1}}h_{\phi L}^p(v)\,
d\widetilde C^\mu_{p,q}(\phi K,\phi Q,v)
=
\int_{S^{n-1}}h_L^p(v)\,
d\widetilde C^{\mu^\phi}_{p,q}(K,Q,v).
\]
Equivalently,
\begin{equation}\label{ii}
\int_{S^{n-1}}h_L^p(v)\,
d\widetilde C^\mu_{p,q}(\phi K,\phi Q,v)
=
\int_{S^{n-1}}h_{\phi^{-1}L}^p(v)\,
d\widetilde C^{\mu^\phi}_{p,q}(K,Q,v).
\end{equation}
Together with \eqref{B}, \eqref{ii}, \eqref{hhom} and \eqref{hphi}, we get
\begin{align*}
&\int_{S^{n-1}}h_L^p(v)\,
d\left(
\phi_{p^\dashv}^t
\widetilde C^{\mu^\phi}_{p,q}(K,Q,\cdot)
\right)(v)\\
&\quad=
\int_{S^{n-1}}
|\phi^{-t}v|^p
h_L^p\bigl(\overline{\phi^{-t}v}\bigr)\,
d\widetilde C^{\mu^\phi}_{p,q}(K,Q,v)\\
&\quad=
\int_{S^{n-1}}
h_L^p(\phi^{-t}v)\,
d\widetilde C^{\mu^\phi}_{p,q}(K,Q,v)\\
&\quad=
\int_{S^{n-1}}
h_{\phi^{-1}L}^p(v)\,
d\widetilde C^{\mu^\phi}_{p,q}(K,Q,v)\\
&\quad=
\int_{S^{n-1}}h_L^p(v)\,
d\widetilde C^\mu_{p,q}(\phi K,\phi Q,v).
\end{align*}
Since
\[
\left\{
c h_L^p-c h_B^p:
L\in\mathcal K_o^n,\ c>0
\right\}
\]
is dense in $C(S^{n-1})$ (see, e.g.,
\cite[Lemma 2.1]{LYZ2018}), we obtain \eqref{SB} for $p\ne0$.

For $p=0$, \eqref{aph} and \eqref{aff} similarly give
\begin{align*}
&\lim_{t\to0}
\frac{
\widetilde D^\mu_q
((1-t)\cdot\phi K+_0t\cdot\phi L,\phi Q,\phi C)
-
\widetilde D^\mu_q(\phi K,\phi Q,\phi C)
}{t}\\
&\quad=
\lim_{t\to0}
\frac{
\widetilde D^\mu_q
\bigl(\phi((1-t)\cdot K+_0t\cdot L),\phi Q,\phi C\bigr)
-
\widetilde D^\mu_q(\phi K,\phi Q,\phi C)
}{t}\\
&\quad=
\lim_{t\to0}
\frac{
\widetilde D^{\mu^\phi}_q
((1-t)\cdot K+_0t\cdot L,Q,C)
-
\widetilde D^{\mu^\phi}_q(K,Q,C)
}{t}.
\end{align*}
By \eqref{vv}, we have
\[
\int_{S^{n-1}}
\log\frac{h_{\phi L}(v)}{h_{\phi K}(v)}
\,d\widetilde C^\mu_{0,q}(\phi K,\phi Q,v)
=
\int_{S^{n-1}}
\log\frac{h_L(v)}{h_K(v)}
\,d\widetilde C^{\mu^\phi}_{0,q}(K,Q,v).
\]
Equivalently,
\begin{align}
&\int_{S^{n-1}}\log h_L(v)\,
d\widetilde C^\mu_{0,q}(\phi K,\phi Q,v)
-
\int_{S^{n-1}}\log h_{\phi K}(v)\,
d\widetilde C^\mu_{0,q}(\phi K,\phi Q,v)
\notag\\
&\qquad=
\int_{S^{n-1}}\log h_{\phi^{-1}L}(v)\,
d\widetilde C^{\mu^\phi}_{0,q}(K,Q,v)
-
\int_{S^{n-1}}\log h_K(v)\,
d\widetilde C^{\mu^\phi}_{0,q}(K,Q,v).
\label{ai}
\end{align}
Taking $L=B$ in \eqref{ai} and using \eqref{hphi}, we have
\[
h_{\phi^{-1}B}(v)
=
h_B(\phi^{-t}v)
=
|\phi^{-t}v|.
\]
It follows that
\begin{align}
&-\int_{S^{n-1}}\log h_{\phi K}(v)\,
d\widetilde C^\mu_{0,q}(\phi K,\phi Q,v)
\notag\\
&\qquad=
\int_{S^{n-1}}\log|\phi^{-t}v|\,
d\widetilde C^{\mu^\phi}_{0,q}(K,Q,v)
-
\int_{S^{n-1}}\log h_K(v)\,
d\widetilde C^{\mu^\phi}_{0,q}(K,Q,v).
\label{BB}
\end{align}
Combining \eqref{ai} and \eqref{BB}, we obtain
\begin{align*}
\int_{S^{n-1}}\log h_L(v)\,
d\widetilde C^\mu_{0,q}(\phi K,\phi Q,v)
&=
\int_{S^{n-1}}
\left(
\log h_{\phi^{-1}L}(v)
-
\log|\phi^{-t}v|
\right)
d\widetilde C^{\mu^\phi}_{0,q}(K,Q,v)\\
&=
\int_{S^{n-1}}
\log h_L\bigl(\overline{\phi^{-t}v}\bigr)\,
d\widetilde C^{\mu^\phi}_{0,q}(K,Q,v),
\end{align*}
where the last equality follows from \eqref{hhom} and \eqref{hphi}, which give
\[
h_{\phi^{-1}L}(v)
=
h_L(\phi^{-t}v)
=
|\phi^{-t}v|\,
h_L\bigl(\overline{\phi^{-t}v}\bigr).
\]
Together with \eqref{B}, this gives
\begin{align*}
&\int_{S^{n-1}}\log h_L(v)\,
d\left(
\phi_{0^\dashv}^t
\widetilde C^{\mu^\phi}_{0,q}(K,Q,\cdot)
\right)(v)\\
&\quad=
\int_{S^{n-1}}
\log h_L\bigl(\overline{\phi^{-t}v}\bigr)\,
d\widetilde C^{\mu^\phi}_{0,q}(K,Q,v)\\
&\quad=
\int_{S^{n-1}}\log h_L(v)\,
d\widetilde C^\mu_{0,q}(\phi K,\phi Q,v).
\end{align*}
Since
\[
\left\{
c\log h_L-c\log h_B:
L\in\mathcal K_o^n,\ c>0
\right\}
\]
is dense in $C(S^{n-1})$ (see, e.g.,
\cite[Lemma 2.1]{LYZ2018}), we obtain \eqref{SB} for $p=0$.
\end{proof}

Combining \eqref{12} and \eqref{23} with $f=h_L^p$, and recalling \eqref{Vpq}, we obtain
\begin{align}\label{dm}
\widetilde{V}_{p,q}^{\mu}(K, L, Q)
&=
\frac{1}{n}\int_{\partial^{\prime}K}
h_L^p(n_K(x))
(x \cdot n_K(x))^{1-p}
\rho_Q^{n-q}(x)g(x)\,
d\mathcal{H}^{n-1}(x)
\notag \\
&=
\frac{1}{n}\int_{S^{n-1}}
\left(
\frac{h_L(\alpha_K(u))}
{h_K(\alpha_K(u))}
\right)^p
\left(
\frac{\rho_K(u)}
{\rho_Q(u)}
\right)^q
\rho_Q^n(u)g(r_K(u))\,du.
\end{align}

In view of \eqref{cp}, for $p\in\mathbb R$, the $L_p$ mixed $\mu$-measure of $K\in\mathcal K_g^n$ and $L\in\mathcal K_o^n$ is given by
\begin{equation*}
V_p^{\mu}(K, L)
:=
\frac{1}{n}\int_{S^{n-1}}
h_L^p(v)\,\mathrm{d}S_p^{\mu}(K, v)
=
\frac{1}{n}\int_{\partial^{\prime}K}
h_L^p(n_K(x))
(x \cdot n_K(x))^{1-p}
g(x)\,d\mathcal{H}^{n-1}(x).
\end{equation*}
For $q<0$, the dual $L_q$ mixed $\mu$-measure was introduced in
\cite{LLH2024}. For sufficiently regular densities, it admits
the following variational representation:
\begin{align*}
\widehat{V}_{q}^{\mu}(K,L)
&:=
\frac{q}{n}
\lim_{t\to 0}
\frac{\mu(L\widetilde{+}_{q}t\cdot K)-\mu(K)}{t}
\\
&=
\frac{1}{n}\int_{S^{n-1}}
\rho_K^q(u)\rho_L^{n-q}(u)
g(r_K(u))\,du.
\end{align*}
This representation generalizes \eqref{wV} to arbitrary measures. Motivated by this formula, we
extend $\widehat{V}_{q}^{\mu}(K,L)$ to all $q\in\mathbb R$ by adopting
the same integral representation. Specifically, for
$K\in\mathcal K_g^n$ and $L\in\mathcal S_o^n$, we define
\begin{equation}\label{08}
\widehat{V}_{q}^{\mu}(K,L)
=
\frac{1}{n}\int_{S^{n-1}}
\rho_K^q(u)\rho_L^{n-q}(u)
g(r_K(u))\,du.
\end{equation}

\begin{thm}
\label{thm:redu}
Let $p,q\in\mathbb R$, $K,L\in\mathcal K_o^n$, and $Q\in\mathcal S_o^n$.
Assume $g\in L^1_{\mathrm{loc}}(\mathbb R^n)$ and that $K$ satisfies
the trace-regularity condition \eqref{radial-trace}. Then,
\begin{align}
\widetilde{V}^{\mu}_{p,q}(K, L, K)
&= \widetilde{V}_{p,n}^{\mu}(K, L, B)
= V^{\mu}_{p}(K, L), \notag \\
\widetilde{V}^{\mu}_{p,q}(K, K, Q)
&= \widetilde{V}_{0,q}^{\mu}(K, L, Q)
= \widehat{V}_{q}^{\mu}(K,Q). \label{K0}
\end{align}
\end{thm}

\begin{proof}
By \eqref{Vpq} and \eqref{s1}, we have
\begin{equation*}
\begin{aligned}
\widetilde{V}_{p,q}^{\mu}(K, L, K)
&= \int_{S^{n-1}} h_L^p(v) \,d\widetilde{C}_{p,q}^{\mu}(K, K, v) \\
&= \frac{1}{n} \int_{S^{n-1}} h_L^p(v)\,dS^{\mu}_p(K, v) \\
&= V_p^{\mu}(K, L),
\end{aligned}
\end{equation*}
and
\begin{equation*}
\begin{aligned}
\widetilde{V}_{p,n}^{\mu}(K,L,B)
&= \int_{S^{n-1}} h_L^p(v) \,d\widetilde{C}_{p,n}^{\mu}(K, B, v) \\
&= \frac{1}{n} \int_{S^{n-1}} h_L^p(v) \,dS^{\mu}_p(K, v) \\
&= V_p^{\mu}(K, L).
\end{aligned}
\end{equation*}

Furthermore, by \eqref{dm} and by applying \eqref{08} with $n-q$ in place of $q$, we obtain
\begin{equation*}
\begin{aligned}
\widetilde{V}^{\mu}_{p,q}(K, K, Q)
&= \widetilde{V}_{0,q}^{\mu}(K, L, Q) \\
&= \frac{1}{n} \int_{S^{n - 1}}
\left( \frac{\rho_K(u)}{\rho_Q(u)} \right)^{q}
\rho_Q(u)^n g(r_K(u)) \,du \\
&= \widehat{V}_{q}^{\mu}(K,Q).
\end{aligned}
\end{equation*}
\end{proof}

\begin{thm}\label{slin}
Let $p,q\in\mathbb R$, $K,L\in\mathcal K_o^n$, and $Q\in\mathcal S_o^n$.
Assume $g\in L^1_{\mathrm{loc}}(\mathbb R^n)$ and that $\phi K$ satisfies
the trace-regularity condition \eqref{radial-trace} for every
$\phi\in\mathrm{SL}(n)$. Then, for every $\phi\in\mathrm{SL}(n)$,
\begin{equation}\label{c}
\widetilde{V}^{\mu}_{p,q}(\phi K, \phi L, \phi Q)
=
\widetilde{V}^{\mu^{\phi}}_{p,q}(K, L, Q).
\end{equation}
\end{thm}

\begin{proof}
First, consider the case $p\neq0$. It follows from \eqref{Vpq}, \eqref{B}, \eqref{SB}, \eqref{hhom}, and \eqref{hphi} that
\begin{align*}
\widetilde{V}^{\mu}_{p,q}(\phi K, \phi L, \phi Q)
&=
\int_{S^{n-1}} h_{\phi L}^{p}(v)
\,d\widetilde{C}^{\mu}_{p,q}(\phi K, \phi Q, v) \\
&=
\int_{S^{n-1}} h_{\phi L}^{p}(v)
\,d\left(
\phi_{p^\dashv}^{t}
\widetilde{C}^{\mu^{\phi}}_{p,q}(K,Q,\cdot)
\right)(v) \\
&=
\int_{S^{n-1}}
\left|\phi^{-t}v\right|^p
h_{\phi L}^{p}\bigl(\overline{\phi^{-t}v}\bigr)
\,d\widetilde{C}^{\mu^{\phi}}_{p,q}(K,Q,v) \\
&=
\int_{S^{n-1}} h_L^p(v)
\,d\widetilde{C}^{\mu^{\phi}}_{p,q}(K,Q,v) \\
&=
\widetilde{V}^{\mu^{\phi}}_{p,q}(K,L,Q).
\end{align*}

Now consider the case $p=0$. By \eqref{K0} and \eqref{08}, we have
\[
\widetilde{V}^{\mu}_{0,q}(K,L,Q)
=
\frac{1}{n}\int_{S^{n-1}}
\left(\frac{\rho_K(u)}{\rho_Q(u)}\right)^q
\rho_Q^n(u)\,g(r_K(u))\,du.
\]
For $u=\phi v/|\phi v|$, equations \eqref{hom} and \eqref{phin} give
\[
\rho_{\phi K}(u)=|\phi v|\rho_K(v),\qquad
\rho_{\phi Q}(u)=|\phi v|\rho_Q(v),\qquad
\mathrm{d}v=|\phi u|^{-n}\,du.
\]
Consequently,
\[
\frac{\rho_{\phi K}(u)}{\rho_{\phi Q}(u)}
=
\frac{\rho_K(v)}{\rho_Q(v)}.
\]
Furthermore,
\[
r_{\phi K}(u)
=
\rho_{\phi K}(u)u
=
|\phi v|\,\rho_K(v)\,\frac{\phi v}{|\phi v|}
=
\phi\bigl(\rho_K(v)v\bigr)
=
\phi\bigl(r_K(v)\bigr).
\]
Therefore,
\begin{align*}
\widetilde{V}^{\mu}_{0,q}(\phi K,\phi L,\phi Q)
&=
\frac{1}{n}\int_{S^{n-1}}
\left(
\frac{\rho_{\phi K}(u)}
{\rho_{\phi Q}(u)}
\right)^q
\rho_{\phi Q}^n(u)
g\bigl(r_{\phi K}(u)\bigr)\,du\\
&=
\frac{1}{n}\int_{S^{n-1}}
\left(
\frac{\rho_K(v)}
{\rho_Q(v)}
\right)^q
\rho_Q^n(v)\,|\phi v|^n
g\bigl(\phi(r_K(v))\bigr)
|\phi v|^{-n}\,dv\\
&=
\frac{1}{n}\int_{S^{n-1}}
\left(
\frac{\rho_K(v)}
{\rho_Q(v)}
\right)^q
\rho_Q^n(v)
g\bigl(\phi(r_K(v))\bigr)\,dv\\
&=
\widetilde{V}^{\mu^\phi}_{0,q}(K,L,Q).
\end{align*}
This completes the proof of \eqref{c}.
\end{proof}

\end{document}